\documentclass[12pt,a4paper]{article}
\usepackage{amsmath}
\usepackage{amssymb}
\usepackage{accents}

\usepackage{mathtools}
\usepackage{amsthm}
\usepackage[usenames]{color}

\usepackage{graphicx}

\usepackage[numbers,sort&compress]{natbib}
\usepackage{doi}   
\usepackage{url} 

\allowdisplaybreaks

\newtheorem{theo}{Theorem}
\newtheorem{theorem}[theo]{Theorem}

\newtheorem{lemma}[theo]{Lemma}

\newtheorem{prop}[theo]{Proposition}

\newtheorem{cor}[theo]{Corollary}
\newtheorem{corollary}[theo]{Corollary}

\newtheorem{remark}[theo]{Remark}

\newtheorem{example}[theo]{Example}

\newcommand{\Z}{{\mathbb Z}}
\newcommand{\N}{{\mathbb N}}

\newcommand{\V}{{\mathcal V}}
\newcommand{\K}{{\mathcal K}}

\newcommand{\R}{{\mathbb R}}

\newcommand{\I}{{\mathcal I}}

\let\phi=\varphi

\newcommand{\bbP}{\mathbb{P}}
\newcommand{\E}{{\mathbb E}}

\newcommand{\gep}{\varepsilon} 

\newcommand{\p}{\mathbb{P}}

\newcommand{\bbS}{\mathbb{S}}

\newcommand{\bL}{\mathbb{L}}

\newcommand{\ind}{{\mathbf{1}}} 

\DeclareMathSymbol{\widehatsym}{\mathord}{largesymbols}{"62}

\newcommand{\capa}{\mathop{\mathrm{cap}}}

\newcommand{\Out}{\mathrm{Out}}
\newcommand{\In}{\mathrm{Ins}}
\newcommand{\Hin}{\mathrm{Hin}}

\newcommand{\harm}{\textrm{harm}}

\title{Random interlacements on transient weighted graphs: 0-1 laws and FKG inequality}
\author{Orph\'ee Collin$^{1}$}

\begin{document}

\maketitle

{\footnotesize 
\noindent $^{~1}$Technische Universit\"at Wien, Austria
\\
\noindent e-mail:
\texttt{orphee.collin@tuwien.ac.at}
}

\begin{abstract}
We study some properties of the random interlacement model on a transient weighted graph, which was introduced by A. Teixeira in \cite{Teixeira2009}. We give a simple proof of the FKG-property and discuss the occurrence of several 0-1 laws for non-local events. We show in particular a 0-1 law for some increasing non-local events, without any assumption.
\\[.3cm]\textbf{AMS subject classifications (2010 MSC):}
60K35, 
82B43, 
60F20, 
60J10. 
\\[.3cm]\textbf{Keywords:} random interlacement model, percolation theory, 0-1 laws
\\[.3cm] This work is licensed under a Creative Commons Attribution | 4.0 International licence (CC-BY 4.0, https://creativecommons.org/licenses/by/4.0/).
\end{abstract}

\tableofcontents

\section{Overview}

We consider an irreducible and reversible Markov chain $(X_n)_{n\ge 0}$ on a countable state space. We use the framework of electrical networks: we consider a simple undirected graph $\mathcal{G}=(\mathbb{V}, \mathcal{E})$ and a collection of positive weights (or conductances) $\mathbf{a}=(a_{x,y})_{(x,y)\in E}$ (with $a_{x, y}=a_{y,x}$), and define on $\mathbb{V}$ the Markov chain $X=(X_n)_{n\ge 0}$ by the following transition probabilities: with $a_x:=\sum_{z\sim x} a_{x,z}$,
\begin{equation}
  p_{x,y}:=
  \begin{cases}
  \frac{a_{x,y}}{a_x}, & \text{ if } (x,y)\in \mathcal{E}, \\
  0 &  \text{ else}.
  \end{cases}
\end{equation}
We furthermore assume that $\mathcal{G}$ is connected and locally finite, and that $X$ is transient.

Following the work of A. Teixeira in \cite{Teixeira2009} (generalizing the construction made in \cite{Sznitman2010} for the SRW in $\Z^d$ in transient dimensions), we consider the random interlacement model built upon this transient weighted graph. The model consists of a random collection of biinfinite trajectories of $X$, which we will refer to as the \emph{interlacement process}: it is indeed a (Poisson) point process on the space of transient nearest-neighbor trajectories on $\mathcal{G}$. 
It is in particular interesting to focus on the \emph{interlacement set} $\I$, defined as the union of the ranges of all the trajectories. This random subset $\I$ of $\mathbb{V}$ can be viewed as a dependent percolation set. It is caracterised by the following property: for every finite subset $K$ of $\mathbb{V}$,
\begin{equation}\label{eq:car_prop_I}
  \bbP(\I \cap K = \emptyset)= \exp(- \capa(K)),
\end{equation}
where $\capa(K)$ denotes the capacity of $K$ with respect to $X$.

In \cite{Teixeira2009}, it was shown that the interlacement set $\I$ satifies the FKG inequality. Here, we point out that this relies in a direct way on the poissonian nature of the interlacement process, and we slightly extend the scope of the FKG inequality.

Next, we study the occurence of a 0-1 law for this model. Contrarily to the case of the SRW on $\Z^d, d\ge 3$, 
there is no natural candidate for an ergodic set of measure-preserving transformations. Our approach to the 0-1 law rather relies on \emph{non-local} events, that is, events that \emph{do not depend on the configuration inside any finite region} (in a sense that will be defined). 

We establish this 0-1 law if $X$ is tail-trivial - that is, when the $\sigma$-algebra of invariant events is trivial with respect to $X$. 
We also give a criterion for the occurence of the 0-1 law when this $\sigma$-algebra is purely atomic.
We further provide a quantitative criterion for the validity of the 0-1 law. 
Additionally, we show that a weaker 0-1 law, concerning events that rely only on one tail of the trajectories, holds in full generality. 
We finally show a 0-1 law for increasing events that are non-local in a slightly stronger sense.


\section{The model}

\subsection{Some notations}

As announced, we consider a simple undirected connected and locally finite graph $\mathcal{G}=(\mathbb{V}, \mathcal{E})$. In the sequel, subsets of $\mathbb{V}$ will be refered to as ``vertex sets''. 
We denote by $\mathcal{P}(\mathbb{V})$ the set of all vertex sets, and by $\mathcal{P}_f(\mathbb{V})$ the set of finite vertex sets. They are endowed with the inclusion order. 
An exhaustion of $\mathbb{V}$ is an increasing sequence $(K_n)_{n\ge 0}$ of finite subsets of $\mathbb{V}$ such that $\bigcup_{n\ge 0} K_n =\mathbb{V}$. 
For $f: P_f(\mathbb{V})\to \R$ and $l\in \R \cup\{-\infty, \infty\}$, 
we say that $f$ converges towards $l$, 
and write $f(K)\underset{K\to\mathbb{V}}\longrightarrow l$, or $\lim_{K\to \mathbb{V}} f(K) = l $,
if for every exhaustion $(K_n)_{n\ge 0}$ of $\mathbb{V}$, $f(K_n)\underset{n \to\infty}\longrightarrow l$. 
We observe that if $f$ is monotonous, then $\lim_{K\to \mathbb{V}} f(K)$ necessarily exists.


\medskip

Furthermore, we consider a family $\mathbf{a}=(a_{x,y})_{(x,y)\in \mathcal{E}}$ of positive conductances on the graph $\mathcal{G}$ and the induced Markov chain $X=(X_n)_{n\ge 0}$ on $\mathbb{V}$, having the following transition probabilities:
\begin{equation}
  p_{x,y}:=
  \begin{cases}
  \frac{a_{x,y}}{a_x}, & \text{ if } (x,y)\in \mathcal{E}, \\
  0 &  \text{ else}.
  \end{cases}
\end{equation}
with $a_x:=\sum_{z\sim x} a_{x,z}$. 
For every vertex $x$, let $P_x$ denote the law of $X$ started from $x$ and $E_x$ the associated expectation. 
We observe that $(a_x)_{x\in \mathbb{V}}$ is an invariant measure for $X$. Throughout the article, we work under the assumption that $X$ is transient.
For every vertex set $K$, we define the two following stopping times:
\begin{equation}
\begin{aligned}
  \tau_K & = \inf \{n\geq 0 : X_n \in K \} \qquad \text{ (hitting time of } K), \\
  \tau^+_K & = \inf \{n\geq 1 : X_n \in K \}  \qquad\text{ (return time to } K),
\end{aligned}
\end{equation}
as well as the random time:
$$\lambda_K := \sup \{n\ge 0 : X_n \in K\}  \qquad \text{ (time of the last visit in } K),$$
We use the conventions $\inf \emptyset = \infty$ and $\sup \emptyset = -\infty$.
We also define the inner escape boundary of $K$:
\begin{equation}
  \partial_X K :=\{x \in K : P_x[\tau_K^+=\infty]>0\}.
\end{equation}
We observe that, by transience, for any $x\in K$, $P_x$-almost surely, $\lambda_K<\infty$ and $X_{\lambda_K}\in \partial_X K$. 

\medskip

We now introduce the notions of equilibrium measure, capacity and harmonic measure of a finite vertex set $K$. The equilibrium measure is the following measure on $\mathbb{V}$, supported on $\partial_X K$: 
\begin{equation}
  e_K(x) := 
  \begin{cases}
    a_x P_x[\tau_K^+=\infty] & \text{if } x\in  K,\\ 
    0 & \text{otherwise}.
  \end{cases}
\end{equation}
The total mass of $e_K$ is called the capacity of $K$ and denoted:
\begin{equation}
  \capa(K):=e_K(\mathbb{V})= \sum_{x\in K} a_x P_x[\tau_K^+=\infty].
\end{equation}
Finally, the harmonic measure of $K$ (or normalized equilibrium measure) is:
\begin{equation}
  \harm_K(\cdot):= \frac{e_K(\cdot)}{\capa(K)}\, .
\end{equation}
Crucial to us is the following property, which derives elementarily from the reversibility of $X$.
For every finite vertex sets $K$ and $L$ with $K\subset L$, we have
\begin{equation}\label{eq:consistency_eq_meas}
  \sum_{x\in L} e_L(x) P_x[X_{\tau_K}\in \cdot, \tau_K<\infty] = e_K(\cdot),
\end{equation}
so that, denoting $P_{\harm_L}(\cdot)=\sum_{x\in \mathbb{V}} \harm_L(x) P_x(\cdot)$, we have
\begin{equation}
  P_{\harm_L}[\tau_K<\infty ]= \frac{\capa(K)}{\capa(L)} 
\end{equation}
and
\begin{equation}
  P_{\harm_L}[X_{\tau_K} \in \cdot |\tau_K<\infty] = \harm_K(\cdot) .
\end{equation}

\subsection{Interlacement process}\label{sec:def_interl_process}

We now present the construction of the interlacement process, refering to \cite{Teixeira2009} for details.

We consider the set of transient nearest-neighbor trajectories in $\mathcal{G}$ indexed by $\N \cup \{0\}$ (respectively, by $\Z$):
\begin{equation}
\begin{aligned}
    W^+=\{ w=(w_n)_{n\ge 0} : \quad  & \forall n\ge 0, w_n \sim w_{n+1} \text{ and }  \\
    & \forall x \in \mathbb{V}, \{n \ge 0 :w_n=x\} \text{ is finite}\},
\end{aligned}
\end{equation}
\begin{equation}
\begin{aligned}
    W=\{ w=(w_n)_{n\in \Z} : \quad  & \forall n\in \Z, w_n \sim w_{n+1} \text{ and }  \\
    & \forall x \in \mathbb{V}, \{n \in \Z:w_n=x\} \text{ is finite}\}.
\end{aligned}
\end{equation}
We endow $W^+$ (respectively, $W$) with the $\sigma$-algebra $\mathcal{W^+}$ (respectively, $\mathcal{W}$) generated by the coordinate maps $w \mapsto w_n$, for $n\ge 0$ (respectively, for $n\in \Z$). Note that the laws $P_x, x\in \mathbb{V}$, are probability measures on $(W^+, \mathcal{W}^+)$.

\smallskip

For every finite vertex set $K$, we denote by $W_K$ the set of trajectories in $W$ that hit $K$ and we consider the measure $\nu_K$ on $(W, \mathcal{W})$ caracterized by:
\begin{equation}
\nu_K\left( \left\{ w\in W: (w_{-n})_{n\ge 0} \in A, w_0=x, (w_n)_{n\ge 0}  \right\} \right) 
 = P_x[A|\tau_K^+=\infty] e_K(x) P_x[B], 
\end{equation}
for every $x\in \mathbb{V}$ and $A, B \in \mathcal{W^+}$. This measure has total mass $\capa(K)$ and is supported by the subset of $W_K$ consisting of the trajectories that enter in $K$ at time $0$.

We may then consider a Poisson point process with intensity measure $\nu_K$. For clarity, let us describe how this Poisson point process can be sampled:
\begin{itemize}
\item sample a number $\xi$ as a Poisson variable of parameter $\capa(K)$,
\item independently for each $i=1,  \dots, \xi$, sample a point $x_i$ under $\harm_K$,
\item independently for each $i=1, \dots, \xi$, attach to $x_i$ a trajectory that enters in $K$ through $x$, at time 0, as follows: independently,
\begin{itemize}
  \item for the negative indices, sample a trajectory under $P_x[ \cdot | \tau_K^+=\infty]$ and apply the function $n\mapsto -n$ to the indices,
  \item for the positive indices, sample a trajectory under $P_x[ \cdot ]$.
\end{itemize}
\end{itemize}
Now, let $K\subset L$ be two finite vertex sets. Consider the Poisson point process with intensity measure $\nu_L$. Restrict it to trajectories that hit $K$, and reindex these trajectories so that they enter $K$ at time 0. It follows from \eqref{eq:consistency_eq_meas} that the obtained point process is Poisson, with intensity $\nu_K$; see Theorem 2.1 in \cite{Teixeira2009} for details.

This consistency property allows us to extend the construction to the whole of $\mathcal{G}$: this gives the random interlacement model on $(\mathcal{G}, \mathbf{a})$. Let us construct it formally. To begin with, due to the reindexation in the consistency property, we choose to look at trajectories in $W$ up to time-translation. We define the equivalence relation:
\begin{equation}
  w\sim w' \quad  \text{ if and only if } \quad \exists k\in \Z, (w_{n+k})_{n\in \Z}=(w_n')_{n\in \Z} .
\end{equation}
Denote by $W^*$ the set of equivalence classes of $W$ under this equivalence relation and by $\pi^*:W\to W^*$ the associated projection. Denote by $\mathcal{W}^*$ the $\sigma$-algebra on $W^*$ inherited from $\mathcal{W}$. For each finite vertex set $K$, denote $W_K^*= \pi^*(W_K)$.

Consider now the measure $\nu$ on $(W^*, \mathcal{W}^*)$ caracterised by the fact that for any finite vertex set $K$, the measure $\nu\ind_{W_K^*}$, i.e, the restriction of $\nu$ to trajectories (seen up to time-shift) that hit $K$, is equal to $\nu_K \circ (\pi^*)^{-1}$. 
One way to construct the measure $\nu$ is to fix an exhaustion $(K_n)_{n\ge 0}$ with $K_0=\emptyset$ and set 
$$\nu = 
\left(\sum_{n\ge 1} \nu_{K_n} \ind_{W_{K_n} \setminus W_{K_{n-1}}} \right)
\circ (\pi^*)^{-1}.$$
Then we consider a Poisson point process $(W^*, \mathcal{W}^*)$ with intensity measure $\nu$. Formally, we define the space:
\begin{equation}\label{eq:def_state_Omega}
  \Omega =  \left\{ \omega = \sum_{i\ge 1} \delta_{w_i^*} ,  w_i^* \in W^* : \, \forall K \in \mathcal{P}_f(\mathbb{V}), \, \omega(W_K^*)<\infty \right\}.
\end{equation}
We endow $\Omega$ with the $\sigma$-algebra $\mathcal{A}$ generated by the maps $\omega \mapsto \omega(D)$ for $D\in \mathcal{W}^*$ 
and we consider on $(\Omega, \mathcal{A})$ the law $\bbP$ of a Poisson point process on $(W^*, \mathcal{W}^*)$ with intensity measure $\nu$. We will refer to $\omega =\sum_{i\ge 1} \delta_{w_i^*} $ (under $\bbP$) as the interlacement process, while 
\begin{equation}
  \I:=\bigcup_{i\ge 1} \text{Range}(w_i^*), \qquad  \V = \mathbb{V} \setminus \I 
\end{equation}
(with obvious definition for $\text{Range}(w_i^*)$) are respectively the interlacement set and the vacant set. 
We recover the caracteristic property of $\I$, announced in \eqref{eq:car_prop_I}, by observing that $\I\cap K=\emptyset$ is equivalent to $\omega(W_K^*)=0$ and that $\omega(W_K^*)$ is distributed as a Poisson variable with parameter $\nu(W_K^*)=\nu_K(W_K)=\capa(K)$. 

\begin{remark}\label{rem:full_interlacement_process}
  It is possible to multiply the intensity measure $\nu$ by a positive real $u$; or equivalently to multiply the family $\mathbf{a}$ of conductances by $u$. 
  We thus get a family of interlacement processes indexed by $u$. 
  There is a natural way to construct all these models together through an increasing coupling: 
  consider a Poisson point process $\sum_{i\ge 1} \delta_{(w_i^*, u_i)}$ on $\mathcal{W}^*\times [0, \infty)$ with intensity measure $\nu \otimes d u$; 
  then, for every $u>0$, set the interlacement process at level $u$ to be  $\omega^u:=\sum_{i\ge 1 : u_i\le u} \delta_{w_i^*}$.
  As soon as $u\le v$, $\omega^u$ is dominated by $\omega^v$. 
  In this paper, we will mostly work with $u$ fixed and equal to 1, and therefore disregard this coupling. 
\end{remark}

\section{Theorems}

\subsection{FKG inequality}


In his seminal paper, A.-S. Sznitman asked whether the interlacement set in $\Z^d$ satifies the FKG inequality (see \cite{Sznitman2010}, Remark 1.6). 
This question was answered positively by A. Teixeira in \cite{Teixeira2009} in the general setup of transient Markov chains, with a detailed proof. 
Here, we point out that this property holds simply because Poisson point processes satisfy in general the FKG inequality. 
We thus show that the interlacement process $\omega$, and not only the interlacement set $\I$, satifies the FKG inequality under $\mathbb{P}$.

\smallskip

Let us recall some definitions.
Let $(\mathcal{X}, \mathcal{F}, \le )$ be a mesurable space equipped with an order. 
A measurable subset $A$ of $\mathcal{X}$ is said to be increasing if 
\begin{equation}
\forall x, y  \in \mathcal{X}, \qquad  ( x \in A \text{ and } x \le y ) \implies y \in A\, .
\end{equation}
If $(\mathcal{X}, \le)$ and $(\mathcal{Y} , \le)$ are two ordered sets, a function $f: \mathcal{X} \to \mathcal{Y}$ is said to be non-decreasing if
\begin{equation}
\forall x, y \in \mathcal{X}, \qquad x \le y \implies f(x) \le f(y)\, . 
\end{equation}
A probability measure $P$ on $(\mathcal{X}, \mathcal{F})$ is said to satisfy the FKG property if for every increasing events $A$ and $B$ we have
\begin{equation}
P[A \cap B ] \ge P[A] \times P[B]\, ,
\end{equation}
or, equivalently, that for every measurable nonnegative and nondecreasing measurable functions $f, g : \mathcal{X} \to \R$, we have
\begin{equation}
E[fg] \ge E[f] \times E[g]\, ,
\end{equation}
where $E$ denotes the expectation under $P$, and $\R$ is equipped with its Borelian $\sigma$-algebra and its usual order. 
An $\mathcal{X}$-valued random variable is said to satisfy the FKG inequality if its law does.

\smallskip

The following lemma is an elementary fact.

\begin{lemma}\label{lem:f(X)_FKG}
Let $(\mathcal{X},\mathcal{F} , \le)$ and $(\mathcal{Y}, \mathcal{G} , \le)$ be two mesurable spaces, equiped with orders. Let $X$ be an $\mathcal{X}$-valued random variable and let $f:\mathcal{X} \to \mathcal{Y}$ be a measurable and non-decreasing function. 
If $X$ satifies the FKG inequality, then so does $f(X)$.
\end{lemma}

We equip the set $\Omega$ introduced in \eqref{eq:def_state_Omega} with the order: 
\begin{equation}
\omega \le \omega'
\qquad 
\text{ if and only if } 
\qquad 
\forall D \in \mathcal{W}^* , \, \omega(D)  \le \omega'(D)\, .
\end{equation}
\qquad 

\begin{theorem}\label{th:FKG}
The law $\mathbb{P}$ on $(\Omega, \mathcal{A})$ satifies the FKG inequality.
\end{theorem}

\begin{proof}
It is a general fact that a Poisson point process on a measurable state space with $\sigma$-finite intensity measure satisfies the FKG inequality. See for example Section 20.3 of \cite{Last2017}. 
\end{proof}

Applying Lemma \ref{lem:f(X)_FKG}, we deduce from Theorem \ref{th:FKG} that the following variables satisfy the FKG inequality:
\begin{itemize}
\item the collection $\eta=(\eta_x)_{x\in \mathbb{V}}\in \{0, 1\}^\mathbb{V}$ defined by 
\[\eta_x := 
\begin{cases}
 1 & \text{if } x 
 \text{ is visited by at least one trajectory in }
 \omega \, , \\
0 & \text{else}  \, ,
\end{cases}
\]
the product set $ \{0, 1\}^{\mathbb{V}}$ being equipped with the product order.
In other terms, $\I$ (and therefore also $\V$) satisfies the FKG inequality (with respect to the inclusion order on subsets of $\mathbb{V}$). This is the FKG property shown in \cite{Teixeira2009}.
\item the collection $\eta=(\eta_x)_{x\in \mathbb{V}} \in \N_0^\mathbb{V}$ defined by 
\[\eta_x := 
\text{ number of trajectories in } \omega \text{ visiting } x \, ,
\]
the set $\N_0:=\{0\} \cup \N$ being equipped with the usual order, and the set $\N_0^\mathbb{V}$ with the product order.
\item the collection $\eta=(\eta_x)_{x\in \mathbb{V}} \in \N_0^\mathbb{V}$ defined by 
\[\eta_x := 
\text{ total number of visits to } x
\text{ by trajectories in } \omega\, .
\]
\end{itemize}


\subsection{The \texorpdfstring{$\sigma$}{sigma}-algebra of non-local events}\label{sec:intro_01laws}

We now turn to the main object of this article, which is to study 0-1 laws for random interlacements in our general framework of transient weighted graphs. 

In the case of vertex- or edge-transitive weighted graphs, there is a natural collection of measure-preserving transformations to be considered and one can expect ergodicity, i.e., a 0-1 law for events which are invariant under these transformations. This is the case for the original work of Sznitamn \cite{Sznitman2010}, i.e., for the SRW on $\Z^d, d\ge 3 $, where it was shown that translation-invariant events have probability either 0 or 1.  
In the general case, we rather choose to consider events that are ``non-local'' in the sense that they do not depend on the restriction of the configuration $\omega$ inside any finite region of the graph $\mathcal{G}$. 

\smallskip 

For every finite vertex set $K$, for every biinfinite trajectory $w\in W_K$, we define two relevant times:
\begin{equation}\label{eq:def_ta_K_lambda_k_w}
\begin{aligned}
\tau_K(w) &  = \inf \{n \in \Z : w_n \in K \} \quad \text{ (hitting time of } K),\\
\lambda_K(w) & = \sup \{n\in \Z : w_n \in K\} \quad  \text{ (time of the last visit in } K),
\end{aligned}
\end{equation}
and we break $w$ in three pieces according to these two times: we denote
\begin{itemize}
  \item by $Ins_K(w)$ the finite piece of trajectory of $w$ between times $\tau_K(w)$ and $\lambda_K(w)$, re-indexed so that it starts at time 0: $Ins_K(w):=(w_{\tau_K(w)+n})_{0\le n \le \lambda_K(w)-\tau_K(w)}$, 
  \item and by $Out_K(w)=(Out_{\rightarrow K}(w) , Out_{K\rightarrow}(w))$ the pair consisiting of the two following semi-infinite pieces of the trajectory $w$:
  \begin{itemize}
  \item $Out_{\rightarrow K}(w)$ the restriction of $w$ up to time $\tau_K(w)$, re-indexed so that it ends at time 0: $Out_{\rightarrow K}(w):=(w_{\tau_K(w)+n})_{n\le 0}$,
  \item and $Out_{K\rightarrow}(w)$ the restriction of $w$ from time $\lambda_K(w)$ on, re-indexed so that it starts at time 0: $Out_{K \rightarrow}(w):=(w_{\lambda_K(w)+n})_{n\ge 0}$.
  \end{itemize} 
\end{itemize}
Furthermore, we define two random variables $\In_K: \omega \mapsto \In_K(\omega)$ and $\Out_K: \omega\mapsto \Out_K(\omega)$ on $(\Omega, \mathcal{A})$ as follows. For any $\omega=\sum_{i\ge 1} \delta_{w_i^*}$ in $\Omega$, we pick for each $i\ge 1$ any representative $w_i$ of $w_i^*$, and define
\begin{equation}
\In_K(\omega)=\sum_{i\in I : w_i^*\in W_K^*} \delta_{Ins_K(w_i)}  \, ,
\end{equation}
and 
\begin{equation}
\Out_K(\omega)=\sum_{i\ge 1: w_i^*\in W_K^*} \delta_{Out_K(w_i)} + \sum_{i\in I : w_i^*\notin W_K^*} \delta_{w_i^*} \, .
\end{equation}
We finally define on $(\Omega, \mathcal{A})$ the sub-$\sigma$-algebra of non-local events:
\begin{equation}
  \mathcal{T}_{RI}= \bigcap_{K \in \mathcal{P}_f(\mathbb{V})} \sigma(\Out_K) \, .
\end{equation}
We will be mainly interested in proving the 0-1 law for the $\sigma$-algebra $\mathcal{T}_{RI}$: it is the property that every event in $\mathcal{T}_{RI}$ has probability either 0 or 1 under $\bbP$. 

\medskip 

This property does not hold in full generality: see Example \ref{ex:simple_not_01} for a simple counter-example.

\begin{example}\label{ex:simple_not_01}
Consider the Markov associated with the nearest neighbor graph on $\Z$ and the family of conductances:
\begin{equation}
  a_{n, n+1}=2^n, \qquad a_{-n-1, -n}=2^n, \qquad  n\ge 0, 
\end{equation}
i.e., the Markov chain with the transition probabilities $p_{0, 1}=p_{0, -1}=\frac{1}{2}$, and, for $n\ge 1$, $p_{n, n+1}=p_{-n, -n-1}=\frac{2}{3}$ and $p_{n, n-1}=p_{-n, -n+1}=\frac{1}{3}$.
Let us consider the trajectories of $\omega$ that go from $-\infty$ to $\infty$. These trajectories have to pass through 0, so their number is distributed as a Poisson variable with parameter 
\begin{equation}
  \begin{aligned}
  P_0& [X_n \underset{n\to\infty}\longrightarrow -\infty |\tau_0^+=\infty ] e_{\{0\}}(0) P_0[X_n \underset{n\to\infty}\longrightarrow \infty]\\
  = & P_0[X_n \underset{n\to\infty}\longrightarrow -\infty, \tau_0^+=\infty ] a_0 P_0[X_n \underset{n\to\infty}\longrightarrow \infty ].
  \end{aligned}
\end{equation}
This parameter is clearly positive and finite
, so the number of trajectories in $\omega$ going from $-\infty$ to $\infty$ is a nontrivial random variable. It is however measureable with respect to $\mathcal{T}_{RI}$. 
\end{example}

More generally, let us give a necessary condition for the 0-1 law to hold.  We consider the sub-$\sigma$-algebra of $(W^+,\mathcal{W}^+)$ consisting of events that are invariant under time-shift. Denoting by 
$\theta: W^+ \to W^+ , \, (w_n)_{n\ge 0}  \mapsto (w_{n+1})_{n\ge 0}$ the time-shift, we set
\begin{equation}
  \mathcal{T}_{MC}:=\left\{ A \in \mathcal{W}^+ : \theta^{-1}(A)=A \right\}.
\end{equation} 
For any events $A, B\in\mathcal{T}_{MC}$, let us consider the set of trajectories $w\in W$ such that $(w_{-n})_{n\ge 0}\in A$ and $(w_n)_{n\ge 0}\in B$ and let us denote by ``$A\rightarrow B$'' its image under the projection $\pi^*$. 
We observe that $\omega(A\to B)$, the random variable counting the number of trajectories in $\omega$ that fall in ``$A\rightarrow B$'', is measurable with respect to $\mathcal{T}_{RI}$ and is distributed as a Poisson variable with parameter $\nu(A\rightarrow B)$. 
The announced necessary condition follows: if the 0-1 law for $\mathcal{T}_{RI}$ holds, then
\begin{equation}\label{eq:necessary_cond}
  \forall A, B\in\mathcal{T}_{MC}, \qquad \nu(A\rightarrow B) \in \{0, \infty\}.
\end{equation}

\subsection{Under tail-triviality or tail-atomicity assumption}\label{sec:state_tail_assumptions}

We give a first condition under which the 0-1 law holds.
We say that the Markov chain $X$ is tail-trivial or that the 0-1 law for $\mathcal{T}_{MC}$ is satisfied if for some $x\in\mathbb{V}$, the invariant $\sigma$-algebra $\mathcal{T}_{RI}$ is trivial under $P_x$, i.e., $P_x$ gives probability either 0 or 1 to all event in $\mathcal{T}_{RI}$. Note that, by irreducibilty, requiring this for some $x$ in $\mathbb{V}$ is equivalent to requiring it for all $x$'s in $\mathbb{V}$. 

\begin{theorem}\label{th:tail_triviality}
Assume that the Markov chain $X$ is tail-trivial. Then, under $\bbP$, every event in $\mathcal{T}_{RI}$ has probability either 0 or 1.
\end{theorem}

Next we go beyond the tail-triviality hypothesis. 
Given a probability space $(\mathcal{X}, \mathcal{J}, \mathbb{Q})$, we say that 
\begin{itemize}
\item an event $A\in \mathcal{J}$ is  atomic if $\mathbb{Q}(A)>0$ and any event included in $A$ has probability either equal to that of $\mathbb{Q}(A)$ or to zero,
\item the probability space is purely atomic if there exists an at most countable collection $(A_i)_{i\in I}$ of atomic events partitioning $\mathcal{X}$.
\end{itemize}
We say that $\mathcal{T}_{MC}$ is purely atomic if $\mathcal{T}_{MC}$ is purely atomic with respect to $P_x$ for some $x$ in $\mathbb{V}$.  Note that, by irreducibilty, requiring this for some $x$ in $\mathbb{V}$ is equivalent to requiring it for all $x$'s in $\mathbb{V}$. 
Our next theorem states that, when $\mathcal{T}_{MC}$ is purely atomic, \eqref{eq:necessary_cond} is a sufficient condition for the 0-1 law to hold.
\begin{theorem}\label{th:tail_atomicity}
Assume that $\mathcal{T}_{MC}$ is purely atomic. 
Assume that for every $A, B\in\mathcal{T}_{MC}$, $\nu(A\rightarrow B)$ is either 0 or infinite. 
Then, under $\bbP$, every event in $\mathcal{T}_{RI}$ has probability either 0 or 1.
\end{theorem}

\begin{remark}
  The conditon \eqref{eq:necessary_cond} can also be formulated as: for every $A, B\in \mathcal{T}_{MC}$, as soon as $P_x(A)>0$ and $P_x(B)>0$ for some $x\in \mathbb{V}$ (equivalently, for all $x$'s in $\mathbb{V}$), the quantity $\nu(A\to B)$ must be infinite.

We observe that
$\nu(A\rightarrow B)=\lim_{K\to \mathbb{V}} \nu(A\rightarrow B, W_K^*)$ and that the quantity appearing in this limit can be expressed as
\begin{equation}
\nu(A\rightarrow B, W_K^*)=\sum_{x\in \partial_X K} e_K(x) P_x[ A | \tau_K^+=\infty] P_x[B].
\end{equation}
\end{remark}

\subsection{Quantitative criterion}

We further provide a quantitative criterion for the occurrence of the 0-1 law for $\mathcal{T}_{RI}$.
For every finite vertex set $K$, let us introduce the hinge measure on $K$, the measure on $\mathbb{V}\times \mathbb{V}$ defined by:
\begin{equation}\label{eq:def_hinge_meas}
h_K(x, y)= e_K(x) \times P_x[X_{\lambda_K}= y]\, ,
\end{equation}
which is supported by $\partial_X K \times \partial_X K$ and symmetric (by reversibility of $X$), whose marginals are both equal to the equilibrium measure $e_K$ and whose total mass is equal to $\capa(K)$.

\begin{theorem}\label{th:quantitative_criterion}
  The 0-1 law for $\mathcal{T}_{RI}$ holds if and only if for every site $x\in \mathbb{V}$, for every $\epsilon>0$, as $L\to \mathbb{V}$,
  \begin{equation}
    \sum_{x', y'\in \partial_X L} h_L(x', y') \left(P_{x'}[\tau_x < \infty | X_{\lambda_L}=y']-\epsilon \right)_+ \to 0.
  \end{equation}
\end{theorem}
\begin{remark}
We stress that, due to \eqref{eq:consistency_eq_meas},
\begin{equation}
    \sum_{x', y'\in \partial_X L} h_L(x', y') P_{x'}[\tau_x < \infty | X_{\lambda_L}=y']
    =  \sum_{x'\in L} e_L(x')  P_{x'}[\tau_x<\infty]
    = Cap(\{x\}).
  \end{equation}
The condition in the theorem can therefore be interpreted as the fact that this sum is ``diluted'' enough, that there is not too much mass supported on some couples $(x',y')$.  

Using the bounds $b \ind_{a > 2b} \le (a-b)_+ \le \ind_{a > b}$ for $a \in [0, 1], b \ge 0$, the condition in the theorem can be equivalently formulated as follows: for every site $x\in \mathbb{V}$, for every $\epsilon>0$, as $L\to \mathbb{V}$,
  \begin{equation}
    \sum_{x', y'\in \partial_X L} h_L(x', y') \ind_{P_{x'}[\tau_x < \infty | X_{\lambda_L}=y'] >\epsilon} \to 0.
  \end{equation}
\end{remark}

\subsection{A weak 0-1 law}

Let us also introduce another $\sigma$-algebra, focusing only on the future pieces of trajectory. For every finite vertex set $K$, we define a random variable $\Out_{K\rightarrow} : \omega \mapsto \Out_{K\rightarrow}(\omega)$ on $(\Omega, \mathcal{A})$ as follows. For any $\omega=\sum_{i\ge 1} \delta_{w_i^*}$ in $\Omega$, we pick for each $i\ge 1$ a representative $w_i$ of $w_i^*$ and denote
\begin{equation}
\Out_{K\rightarrow}(\omega)=\sum_{i\in I : w_i^*\in W_K^*} \delta_{Out_{K\rightarrow}(w_i)} + \sum_{i\in I :  w_i^*\notin W_K^*} \delta_{w_i^*} \, .
\end{equation}
The $\sigma$-algebra of non-local future events is defined as
\begin{equation}
  \mathcal{T}_{RI\rightarrow} := \bigcap_{K\in \mathcal{P}_f(\mathbb{V})} \sigma(\Out_{K\rightarrow}) \, .
\end{equation}
This $\sigma$-algebra is a priori not so interesting because it lacks much information. We mention it however, since it has the nice property that it satisfies the 0-1 law in full generality. 
\begin{theorem}\label{th:weak_01law}
Under $\bbP$, every event in $\mathcal{T}_{RI\rightarrow}$ has probability either 0 or 1.
\end{theorem}

\begin{remark}\label{rem:backward_weak}
  By reversibility of the Markov chain $X$, reversing time for all the trajectories in $\omega$ does not change the law of $\omega$ under $\bbP$.
  Thus, we obtain also the 0-1 law for the $\sigma$-algebra of non-local past events, defined by focusing on the past pieces of trajectory $Out_{\rightarrow K}(w)$ instead of the future pieces $Out_{K\rightarrow}(w)$.
\end{remark}

\subsection{A 0-1 law for increasing non-local events}

We now present two other ways to weaken the 0-1 law discussed in Section \ref{sec:intro_01laws}, that may lead to 0-1 laws holding under more general assumptions.

\smallskip

First, one could ask only for the 0-1 law to hold for increasing events belonging to $\mathcal{T}_{RI}$. We observe that condition \eqref{eq:necessary_cond} remains necessary for this modified $0-1$ law to hold. Indeed, if there exist events $A$ and $B$ in $\mathcal{T}_{MC}$ such that $\nu(A\to B) \in (0, \infty)$ then for any positive integer $k$, the event $\{\omega(A\to B ) \ge k\}$ is an increasing event in $\mathcal{T}_{RI}$ which has probability different from 0 and from 1.
So this modified 0-1 law does not hold in full generality, and, in the case where $\mathcal{T}_{MC}$ is tail-atomic, it is in fact equivalent to the 0-1 law for $\mathcal{T}_{RI}$. 

\smallskip

Secondly, we consider the possibility of forgetting the pairings of the pieces of trajectories when defining the trace of configuration $\omega$ outside $K$.
We set, for $\omega=\sum_{i\ge 1} \delta_{\omega_i^*}$, picking for each $i\ge 1$ any representative $w_i$ of $w_i^*$,
\begin{equation}
\widetilde{\Out}_K(\omega):=\sum_{i\ge 1: w_i^*\in W_K^*} \delta_{Out_{\to K}(w_i)} + \sum_{i\ge 1: w_i^*\in W_K^*} \delta_{Out_{K\to}(w_i)} + \sum_{i\in I : w_i^*\notin W_K^*} \delta_{w_i^*} \, 
\end{equation}
and define 
$$\widetilde{\mathcal{T}}_{RI} := \bigcap_{K \in \mathcal{P}_f(\mathbb{V})} \sigma\left(\widetilde{\Out}_K\right). $$
Once again, the 0-1 law for $\widetilde{\mathcal{T}}_{RI}$ does not hold in full generality. For example, we observe that if there exists an event $A\in \mathcal{T}_{MC}$ such that, with $ A^c= W^+\setminus A$, we have $\omega(A\to  A^c)\in (0, \infty)$, then the variable 
$$\omega(A \to  W^+)-\omega( W^+ \to A) = \omega(A \to  A^c)-\omega( A^c \to A)$$ 
is mesurable with respect to $\widetilde{\mathcal{T}}_{RI}$ and non-trivial. 

\medskip

Our final theorem states that, remarkably, combining the two approaches we just explored yields a 0-1 law that holds in full generality.

\begin{theorem}\label{th:increasing_01law}
  Under $\p$, any increasing event in $\widetilde{\mathcal{T}}_{RI} $ has probability either 0 or 1.
\end{theorem}

\begin{remark}
  Of course, by taking complements, also decreasing events in $\widetilde{\mathcal{T}}_{RI} $ have probability 0 or 1 under $\p$.
\end{remark}

\section{Preliminaries for the proof of the 0-1 laws}

\subsection{Hinge decomposition of the interlacement process}\label{sec:hinge_decomp}

Given a finite vertex set $K$ and a trajectory $w\in W_K$, recall the definition of the times $\tau_K(w)$ and $\lambda_K(w)$ in \eqref{eq:def_ta_K_lambda_k_w} and define the hinge couple of $w$ as:
\begin{equation}
  Hin_K(w):= (w_{\tau_K(w)}, w_{\lambda_K(w)}) .
\end{equation}
We define the \emph{hinge process} $\Hin_K(\omega)$ associated to $\omega=\sum_{i\ge 1} \delta_{w_i^*}$ by picking for each $i \ge 1$ any representative $w_i$ of $w_i^*$ and setting:
\begin{equation}
  \Hin_K(\omega)=\sum_{i\ge 1 : w_i^*\in W_K^*} \delta_{Hin_K(w_i)}.
\end{equation}
Under $\p$, the random variable $\Hin_K$ is distributed as a Poisson point process with intensity measure the hinge measure $h_K$, introduced in \eqref{eq:def_hinge_meas}. 

Let us describe how to sample $\omega$ under $\p$ with respect to the decomposition along the hinge process. 
To sample the trajectories touching $K$,
\begin{itemize}
\item sample first the hinge process, as a Poisson point process with intensity $h_K$.
\item to each hinge couple $(x, y)$, attach a trajectory $w^*$ which decomposes as $Ins_K(w)$ and $Out_K(w)=(Out_{K\rightarrow} (w), Out_{\rightarrow K}(w))$:
\begin{itemize}
\item $Ins_K(w)$ is sampled according to $P_x[ (X_0, X_1, \dots, X_{\lambda_K})\in \cdot | X_{\lambda_K}=y]$ 
\item $Out_{K\rightarrow} (w)$ is obtained by applying a time reversal to a trajectory sampled according to $P_x[ \cdot |  \tau^+_K=\infty]$,
\item $Out_{\rightarrow K}(w)$ is sampled according to $P_y[ \cdot | \tau^+_K=\infty]$.
\end{itemize} 
\end{itemize}
Complete $\omega$ by sampling the trajectories that do not touch $K$, as usual as a Poisson point process with intensity measure $\nu$ restricted to $W^*\setminus W_K^*$.

\smallskip

We observe that the interlacement process satisfies a spatial Markov property: under $\p$, $\Out_K$ depends on $\In_K$ only through the hinge process $\Hin_K$. Let us thus describe how to go from a finite vertex set K to a larger finite vertex set L: this will be at the heart of our arguments for the 0-1 law.
To sample $\In_L$ given $\In_K$, independently
\begin{itemize}
  \item extend each path in $\In_K$, say with extremities $(x,y)$, by adding two pieces of trajectory:
\begin{itemize}
  \item on the left, one obtained by applying a time reversal to a trajectory sampled according to $P_x[ (X_0, X_1, \dots, X_{\lambda_{L}})\in \cdot |  \tau_{K}^+=\infty]$,
  \item on the right, one sampled according to $P_y[ (X_0, X_1, \dots, X_{\lambda_{L}})\in \cdot |  \tau_{K}^+=\infty]$,
\end{itemize}
\item sample additional trajectories inside $L$ that do not touch $K$ by sampling biinfinite trajectories under $\nu$ restricted to $W_L^*\setminus W_K^*$ and keeping only the pieces of trajectories inside $L$.
\end{itemize}


\subsection{Coupling results}

For $\mu$ and $\nu$ probability measures on an at most countable state space $\mathcal{X}$ (equipped with the $\sigma$-algebra consisting of all its subsets), we denote by $d_{TV}(\mu, \nu)$ their total variation distance, which is defined as
\begin{equation}\label{eq:def_dTV}
  \begin{aligned}
    d_{TV}(\mu, \nu) : & = \sup_{A\subset \mathcal{X}} |\mu(A)- \nu(A)| \\
    & = \sum_{x\in \mathcal{X} : \mu(x)>\nu(x) } \mu(x)- \nu(x) \\
    & = \frac{1}{2} \sum_{x\in \mathcal{X} } |\mu(x)- \nu(x) |\\
    & = \inf_{(X, Y) : X\sim \mu, Y\sim \nu} P(X \neq Y) .
   \end{aligned}
\end{equation}
The above infimum is over couplings $(X, Y)$ of $\mu$ and $\nu$ and it is achieved.

\medskip
Using a classical approach to 0-1 laws, we now establish a sufficient condition for the 0-1 law for $\mathcal{T}_{RI}$ in terms of total variation distances. 
For any finite vertex set $K$, 
\begin{itemize}
\item we call admissible path inside $K$ any nearest-neighbor path in $\mathcal{G}$ with finite length, with starting time 0 and with extremities belonging to $\partial_X K$.
\item we call admissible configuration inside $K$ any configuration $\eta$ such that $\p(\In_K=\eta)>0$, that is, any configuration of the form $\sum_{i=1}^k \delta_{\gamma_i}$, where $k$ is a finite integer (possibly zero), and for each $i=1, \dots, k$, $\gamma_i$ is an admissible path inside $K$. 
\end{itemize}

\begin{lemma}  \label{lem:suf_cond_0_1_law_distance}
  Assume that for every finite vertex set $K$, for every admissible configuration $\eta$ inside $K$, as $L\to \mathbb{V}$,
  \begin{equation}\label{eq:condition_dTV_to0_lem_one_dir}
    d_{TV}\left(\mathbb{P}(\Out_L  \in \cdot | \In_K =\eta),
    \mathbb{P}(\Out_L  \in \cdot | \In_K =0) \right) \to 0.
  \end{equation}
  Then every event in $\mathcal{T}_{RI}$ has probability either 0 or 1 under $\p$.

  Furthermore, the above condition can be restricted to $\eta$'s consisting of a single path.
\end{lemma}
\begin{remark}\label{rem:rewrite_d_TV_cond}
We remark that, due to the Markovian property evoked in Section \ref{sec:hinge_decomp}, we can rewrite the total variation distance in \eqref{eq:condition_dTV_to0_lem_one_dir} as
\begin{equation}
   d_{TV}\left(\mathbb{P}( \Hin_L  \in \cdot | \Hin_K =\rho),
    \mathbb{P}(\Hin_L  \in \cdot | \Hin_K =0) \right),
\end{equation}
  where $\rho$ is the hinge process on $K$ corresponding to $\eta$. 
  So the condition in Lemma \ref{lem:suf_cond_0_1_law_distance} really concerns the dependance of the hinge process $\Hin_L$ on the hinge process $\Hin_K$. 
\end{remark}

\begin{proof}
  We assume that the condition in the lemma holds. Let $E\in \mathcal{T}_{RI}$. Let $K$ be a finite vertex set and $\eta$ an admissible configuration on $K$, and let us consider, under a probability $\mathbf{P}$, a coupling $(\omega, \omega')$ with $\omega$ distributed under  $\mathbb{P}(\cdot | \In_K=\eta)$ and $\omega'$ distributed under $\mathbb{P}(\cdot | \In_K=0)$, such that $\Hin_L(\omega)=\Hin_L(\omega')$ with maximal probability, i.e.,
  \begin{equation}
    \mathbf{P}[\Hin_L(\omega) \neq \Hin_L(\omega')] =  
    d_{TV}\left(\mathbb{P}(\Hin_L \in \cdot | \In_K=\eta),
    \mathbb{P}(\Hin_L \in \cdot | \In_K=0) \right) 
  \end{equation}
  and such that on the event $\{\Hin_L(\omega)=\Hin_L(\omega')\}$, $\omega$ and $\omega'$ coincide outside $L$, in the sense that $\Out_L(\omega)= \Out_L(\omega')$. Since $E$ belongs to $\sigma(Out_L)$, on the latter event $\omega$ realises $E$ if and only if $\omega'$ does. It follows that:
  \begin{equation}
    \begin{aligned}
    & |\p(E|\In_K = \eta)  - \p(E|\In_K  = 0)| \\
    & = |\mathbf{E}[ \ind_{\omega \in E} - \ind_{\omega' \in E} ]| \\
    & \le \mathbf{P}[ \ind_{\omega \in E} \neq  \ind_{\omega' \in E} ] \\
    & \le \mathbf{P}(\Hin_L(\omega ) \neq \Hin_L(\omega'))\\
    & = d_{TV}\left(\mathbb{P}(\Hin_L \in \cdot | \In_K=\eta),
    \mathbb{P}(\Hin_L \in \cdot | \In_K=0) \right) .
    \end{aligned}
  \end{equation}
  Sending $L$ to $\mathbb{V}$, the hypothesis of the lemma yields $\p(E|\In_K= \eta) = \p(E|\In_K=0)$. 
  Since this is true for any admissible $\eta$, we get that, under $\p$, $E$ is independent of the variable $\In_K$.  
  Since the variables $(\In_K, K\in \mathcal{P}_f(\mathbb{V}))$ generate the $\sigma$-algebra $\mathcal{A}$, this implies that, under $\p$, $E$ is independent of $\mathcal{A}$, and in particular of itself, hence $\p(E) \in \{0 , 1\}$. This concludes the law of the 0-1 law.

  \smallskip 

  To justify that one can restrict to $\eta$'s of the form $\eta=\delta_\gamma$, where $\gamma$ is an admissible path inside $\K$, we observe that for every admissible configuration $\eta$ inside $K$ and every admissible path $\gamma$ inside $K$,
\begin{equation}
  \begin{aligned}
  & d_{TV}\left(\mathbb{P}(\Hin_L  \in \cdot | \In_K =\eta+\delta_\gamma),
    \mathbb{P}(\Hin_L  \in \cdot | \In_K =\eta) \right) \\
  & \leq d_{TV}\left(\mathbb{P}(\Hin_L  \in \cdot | \In_K =\delta_\gamma),
    \mathbb{P}(\Hin_L  \in \cdot | \In_K =0) \right) .
  \end{aligned}
\end{equation}
  Indeed, it is just a matter of coupling identically the trajectories corresponding to $\eta$ in both processes. 
  The condition in the lemma derives by reasoning by induction on the number of paths in $\eta$, with the help of the triangular inequality for $d_{TV}$.
\end{proof}

In order to establish the quantitative criterion of Theorem \ref{th:quantitative_criterion}, we will show in Section \ref{sec:necessary_and_sufficient_for_quantitative} that the condition of Lemma \ref{lem:suf_cond_0_1_law_distance} is also necesssary for the 0-1 law for $\mathcal{T}_{RI}$ to hold. This will result from a more general statement that we present now, and that will also be useful (in Section \ref{sec:prelim_tail_assump}) to exploit the tail-triviality or tail-atomicity assumptions of Section \ref{sec:state_tail_assumptions}. 

\smallskip

Let us first recall a useful fact. 
For a (possibly time-inhomogeneous) Markov chain $Z=(Z_n)_{n\ge 0}$ on a countable state space $\mathcal{X}$, 
for any starting points $x$ and $y$ in $\mathcal{X}$, there exists, under a probability measure $\mathbf{P}$, a coupling $(X, Y) = ((X_n)_{n\ge 0} , (Y_n)_{n\ge 0} ) $ of the Markov chain started respectively at $x$ and at $y$ 
such that, almost surely, $X$ and $Y$ coincide after their first meeting time $T:=\inf\{n\ge 0: X_n=Y_n\}$ 
and such that for every $n\ge 0$:
\begin{equation}\label{eq:def_max_coupl}
  \mathbf{P}[T>n]=d_{TV}(P_x[Z_n \in \cdot], P_y[Z_n \in \cdot]).
\end{equation}
This coupling is called the maximal coupling of the Markov chain $Z$, see \cite{Griffeath1975, Pitman1976} or Section 12.4.2 of \cite{Lawler2010}. 

\begin{lemma}\label{lem:general_equiv_01law}
Consider a (possibly time-inhomogeneous Markov chain) $(Z_n)_{n\ge 0}$ on a discrete countable state space. We fix $x_0\in \mathcal{X}$. The following statements are equivalent:
\begin{itemize}
  \item[(i)] For every $n\ge 0$, for every $x$ and $y$ in $\mathcal{X}$ such that $P_{x_0}(Z_n =x)>0$ and $P_{x_0}(Z_n =y)>0$, we have
  \begin{equation}
  d_{TV}(P_{x_0}(Z_m \in \cdot | Z_n = x), P_{x_0}(Z_m \in \cdot | Z_n = y))  \underset{m\to\infty}{\longrightarrow} 0.
  \end{equation}
  \item[(ii)] Every event in $\mathcal{T}:=\bigcap_{n\ge 0} \sigma(Z_m , m \ge n)$ has probability either 0 or 1 under $P_{x_0}$.
\end{itemize} 
\end{lemma}
\begin{proof} The proof of the implication $(i)\implies (ii)$ follows the same scheme as the proof of Lemma \ref{lem:suf_cond_0_1_law_distance}.
  We assume that $(i)$ holds and consider $E\in\mathcal{T}$. 
  We fix $n\ge 0$ and $x,y$ in $\mathcal{X}$ such that $P_{x_0}(Z_n =x)>0$ and $P_{x_0}(Z_n =y)>0$. 
  Let us consider now an integer $m \ge n$. We use, under a probability $\mathbf{P}$, a coupling $(X, Y)$ of $P_{x_0}(\cdot | Z_n =x)$ and $P_{x_0}(\cdot | Z_n =y)$ such that $X_m= Y_m$ with maximal probability, i.e.:
  \begin{equation}
    \mathbf{P}[X_m\neq Y_m]  =  d_{TV}(P_{x_0}(Z_m \in \cdot | Z_n = x), P_{x_0}(Z_m \in \cdot | Z_n = y)).
  \end{equation}
  We further require that, on the event $X_m=Y_m$, $X$ and $Y$ coincide after time $m$. Since $E \in \sigma(Z_k, k\ge m)$, on the latter event $X$ realises $E$ if and only if $Y$ does. It follows that:
  \begin{equation}
    \begin{aligned}
    & |P_{x_0}(E | Z_n = x) - P_{x_0}(E | Z_n = y)| \\
    & = |\mathbf{E}[ \ind_{X \in E} - \ind_{Y \in E} ]| \\
    & \le \mathbf{P}[ \ind_{X \in E} \neq  \ind_{Y \in E} ] \\
    & \le \mathbf{P}[X_m \neq Y_m]\\
    & = d_{TV}(P_{x_0}(Z_m \in \cdot | Z_n = x), P_{x_0}(Z_m \in \cdot | Z_n = y))\, .
    \end{aligned}
  \end{equation}
  Sending $m$ to infinity, condition $(i)$ yields that $P_{x_0}(E | Z_n = x) = P_{x_0}(E | Z_n = y) $. 
  We have thus established that, under $P_{x_0}$, $E$ is independent of the variable $Z_n$, for every $n \ge 0$. It follows that, under $P_{x_0}$, $E$ is independent of $\sigma(Z_n , n\ge 0)$, and in particular of itself, hence $P_{x_0}(E) \in \{0 , 1\}$ and $(ii)$ is proved. 

  \smallskip

  Conversely, let us prove the implication $(ii)\implies (i)$. Let us assume that the 0-1 law holds, let us consider $n\ge 0$ and $x,y  \in \mathcal{X}$ such that $P_{x_0}(E | Z_n = x)>0$ and $ P_{x_0}(E | Z_n = y) >0$ and let us show that
  \begin{equation}\label{eq:def_delta}
    \delta :=  \lim_{m\to\infty} d_{TV}(P_{x_0}(Z_m \in \cdot | Z_n = x), P_{x_0}(Z_m \in \cdot | Z_n = y)) 
  \end{equation}
  is zero. Note that the quantity in the limit is non-increasing by elementary arguments.
  Let us consider, under a probability $\mathbf{P}$, $(X,Y)=((X_m)_{m\ge n},(Y_m)_{m\ge n} )$ the maximal coupling of the Markov chain started at time $n$ respectively from $x$ and $y$, in the sense of \eqref{eq:def_max_coupl}. For each $m\ge n$, we consider the subset $E_m$ of $\mathcal{X}$ defined as:
  \[E_m:=\{z \in \mathcal{X}, P_{x_0}(Z_m = z | Z_n = x) > P_{x_0}(Z_m = z | Z_n = y)  \} \,, \]
  which satisfies 
  \begin{equation}\label{eq:PEm_ge_delta}
  P_{x_0}(Z_m \in E_m | Z_n = x) 
   \ge P_{x_0}(Z_m \in E_m | Z_n = y)  
   + \delta 
   \ge \delta,
  \end{equation}
  by the definition of $\delta$ in \eqref{eq:def_delta} and by the definition of the total variation distance in \eqref{eq:def_dTV}. 
  Let us consider the event 
  $E:=\{Z_m \in E_m \text{ for infinitely many } m \text{'s}\}$. 
  This event belongs to $\mathcal{T}$, thus $(i)$ yields that $p:=P_{x_0}(E)$ is either 0 or 1. We observe that then also $P_{x_0}( E | Z_n = x)$ and $P_{x_0}(E | Z_n = y)$ are equal to $p$. 
  
  On the one hand, if $p=0$, then $P_{x_0}( Z_m \in E_m \text{ i.o.} |Z_n=x )=0$, which implies that $P_{x_0}(Z_m \in E_m |Z_n=x )  \underset{m\to\infty}\longrightarrow 0$, and then \eqref{eq:PEm_ge_delta} yields that $\delta=0$. 

  On the other hand, if $p=1$, then $\mathbf{P}(Y_m \in E_m \text{ i.o.} ) = P_{x_0}(E|Z_n=y)=1$ so, $\mathbf{P}$-almost surely, there exists an $m\ge n$ (in fact, infinitely many $m$'s) such that $Y_m \in E_m$.
  By maximality of the coupling, 
  $Y_m\in E_m$ implies almost surely that $T\le m$. Thus, $T$ is almost surely finite and $\delta= \lim_{m\to \infty} \mathbf{P}(T>m)=0$. 
\end{proof}

We finally provide an estimate on the total variation distance between $Poi(\lambda)$, which denotes the Poisson distribution with parameter $\lambda$ and $Poi(\lambda)+1$, which denotes the law of $X+1$, where $X$ follows $Poi(\lambda)$.
\begin{lemma}\label{lem:dTV_Poisson}
For every $\lambda > 0$, 
\begin{equation}
  d_{TV}(Poi(\lambda), Poi(\lambda)+1) \le \frac{1}{2\sqrt{\lambda}}.
\end{equation}
\end{lemma}

\begin{proof}
With $X\sim Poi(\lambda)$, 
\begin{equation}
  \begin{aligned}
    2 \, d_{TV}(Poi(\lambda), Poi(\lambda)+1)  
    & = \sum_{k\ge 0} \frac{\lambda^k}{k!} e^{-\lambda} \left|\frac{k}{\lambda}-1 \right|\\
    & = E\left[\left|\frac{X}{\lambda} - 1 \right|\right]\\ 
    & \le E\left[\left(\frac{X}{\lambda} - 1 \right)^2 \right]^{\frac{1}{2}}
     = \frac{1}{\sqrt{\lambda}}\, .
  \end{aligned}
\end{equation}
\end{proof}


\section{Proofs under assumptions on \texorpdfstring{$\mathcal{T}_{MC}$}{TMC}}

In this section, we prove Theorems \ref{th:tail_triviality} and \ref{th:tail_atomicity}.

\subsection{Preliminaries}\label{sec:prelim_tail_assump}

We first show that the total mass of $\nu$ is infinite, so that the total number of trajectories in the interlacement process is almost surely infinite.

\begin{lemma}\label{lem:cap_to_infty}
  As $K\to\mathbb{V}$, $\capa(K)\to \infty$. 
\end{lemma}
\begin{proof}
  We use potential theory for the irreducible and reversible Markov chain $X$ (cf Sections 2.4 and 2.5 in Lyons and Peres). 
  We recall that, for any finite vertex set $K$, a unit flow from $K$ to infinity is a function $\theta:\mathcal{E} \to [0, \infty)$ such that:
  \begin{itemize}
  \item $\forall (x,y)\in \mathcal{E}, \theta(x,y)=-\theta(y,x),$
  \item $\forall x \in \mathbb{V}\setminus K, \sum_{y\sim x } \theta(x,y)=0,$
  \item $\forall x, y \in K$ such that $(x,y)\in \mathcal{E}$, $\theta(x,y)=0$,
  \item $ \sum_{x\in K, y\notin K: x\sim y} \theta(x,y)=1$,
  \end{itemize}
  and that we have
  \begin{equation}
    \frac{2}{\capa(K)} = 
    \inf \left\{
      \sum_{x,y\in \mathbb{V}: x\sim y} \frac{1}{a_{x,y}} \theta(x,y)^2 : 
    \theta \text{ is a unit flow from } K \text{ to infinity}
    \right\}, 
  \end{equation}
  with the convention $1/0=\infty$. 
  Let us fix a non-empty vertex set $K_0$, for example a singleton. By transience of $X$, $\capa(K_0)>0$, hence there exists a unit flow $\theta_0$ from $K_0$ to infinity such that $\sum_{x,y\in \mathbb{V}: x\sim y} \frac{1}{a_{x,y}} \theta_0(x,y)^2<\infty$. 
  Then, for every vertex set $K$ containing $K_0$, considering
  \begin{equation}
  \theta_K(x,y):= 
    \begin{cases}
      0 & \text{if } x, y \in K, \\
      \theta_0(x,y)  & \text{else}.
    \end{cases}
  \end{equation}
  which is a unit flow from $K$ to infinity, we get that 
  \begin{equation}
    \frac{2}{\capa(K)}  \le \sum_{(x,y) \in \mathbb{V}^2 \setminus K^2 : x\sim y} \frac{1}{a_{x,y}} \theta_0(x,y)^2
  \end{equation}
  and the latter vanishes as $K\to\mathbb{V}$. The lemma follows.
\end{proof}

When the Markov chain $X$ is tail-trivial, we produce a coupling of $X$ started at two different vertices such that the two copies coincides eventually, up to time-shift.

\begin{lemma}\label{lem:equiv_conditions_01_MC}
  Assume that the Markov chain $X$ satisfies the 0-1 law for $\mathcal{T}_{MC}$. Then, for every $x, y\in \mathbb{V}$, there exists a coupling $X=(X_n)_{n\ge 0} , Y=(Y_n)_{n\ge 0}$ of the Markov chain $X$ started respectively at $x$ and $y$ such that, almost surely, there exist (random) finite times $S$ and $T$ for which
  \begin{equation}
    \forall n\ge 0 , \qquad  X_{S+n}=Y_{T+n}, 
  \end{equation} 
  that is, $X$ and $Y$ coincide eventually, up to time-shift. 
\end{lemma}

\begin{proof}
  Assume that the Markov chain $X$ satisfies the 0-1 law for $\mathcal{T}_{MC}$ and pick $x$ and $y$ in $\mathbb{V}$. 
  Consider an exhaustion $(K_n)_{n\ge 0}$ of $\mathbb{V}$ such that $K_0$ contains $x$ and $y$. 
  The Markov chain $(X_{\lambda_{K_n}})_{n\ge 0}$ is tail-trivial and an application of Lemma \ref{lem:general_equiv_01law} yields:
  \begin{equation}\label{eq:proof_MC_coupling_application_Lemma}
  d_{TV}(P_x[X_{\lambda_{K_n}} \in \cdot], P_y[X_{\lambda_{K_n}} \in \cdot]) \underset{n\to  \infty}\longrightarrow 0
  \end{equation}
  Consider the maximal coupling of the Markov process $(X_{\lambda_{K_n}})_{n\ge 0}$, started from $P_x[X_{\lambda_{K_0}}\in \cdot ]$ and $P_y[X_{\lambda_{K_0}}\in \cdot ]$, by \eqref{eq:proof_MC_coupling_application_Lemma}, the two processes almost surely coincide eventually. Now, given the realisation of this coupling, sample trajectories of the Markov chain $X$, started respectively from $x$ and $y$, conditionally on the exit points prescribed by the coupling. Where the two copies of $(X_{\lambda_{K_n}})_{n\ge 0}$ agree, build the two trajectories of $X$ identically. Then, almost surely, the two trajectories eventually coincide, up to time-shift.
\end{proof}

\begin{corollary}\label{cor:coincide_up_time_shift}
  Assume that the 0-1 law for $\mathcal{T}_{MC}$ is satisfied. 
  Then, for every finite vertex set $K$, for every $x, y\in \partial_X K$, there exists a coupling $(X, Y)$ of $P_x[\cdot | \tau_K^+=\infty]$ and $P_y[\cdot | \tau_K^+=\infty]$ such that, almost surely, there exist (random) finite times $S$ and $T$ for which
  \begin{equation}
    \forall n\ge 0 , \qquad  X_{S+n}=Y_{T+n}, 
  \end{equation} 
  that is, $X$ and $Y$ coincide eventually, up to time-shift. 
\end{corollary}
\begin{proof}
  Consider the Markov chain on $\{x\in \mathbb{V}: P_x(\tau_K=\infty)>0\}$ obtained by conditioning the Markov chain $X$ on the event $\tau_K=\infty$. This Markov chain inherits the tail-triviality of $X$ and Lemma \ref{lem:equiv_conditions_01_MC} applies. 
  Now, to build the coupling in the corollary, just choose any way to couple the first step of $X$ and $Y$, and starting from $X_1$ and $Y_1$ couple $X$ and $Y$ such that they coincide eventually up to time-shift. 
\end{proof}

\subsection{Proof in the case of tail-triviality}\label{sec:proof_tail_triviality}

We now prove Theorem \ref{th:tail_triviality}. The proof follows the idea of \cite{Collin2024}.

\begin{proof}[Proof of Theorem \ref{th:tail_triviality}]
  We assume that the 0-1 law for $\mathcal{T}_{MC}$ holds and we show that the condition of Lemma \ref{lem:criterion_dTV_to_0_law} is satisfied. Let $K$ be a finite vertex set and $\eta$ an admissible configuration inside $K$ consisting of a single path. 
  Let $\gep>0$. 
  Our aim is to find a finite vertex set $M$ and to construct a coupling $(\omega, \omega')$ of $\p( \cdot | \Hin_K =\eta)$ and $\p( \cdot | \Hin_K = 0 )$ such that $\Out_M(\omega)=\Out_M(\omega')$ with probability not smaller than $1-\epsilon$.

  By Lemma \ref{lem:cap_to_infty}, we pick $L$ a vertex set containing $K$ and such that 
  \begin{equation}\label{eq:capa_L-capa_Kgeeps}
  \capa(L)-\capa(K)\ge \epsilon^{-2}.
  \end{equation}
  We observe that the number of trajectories in $\omega$ hitting $L$ is distributed as $Poi(\capa(L)-\capa(K))+1$, while for $\omega'$, this number is distributed as $Poi(\capa(L)-\capa(K))$. Using \eqref{eq:capa_L-capa_Kgeeps} and Lemma \ref{lem:dTV_Poisson}, we construct $\In_L(\omega)$ and $\In_L(\omega')$ such that these numbers coincide with probability not smaller than $1- \epsilon / 2$. 
 
  Now, given $\In_L(\omega)$ and $\In_L(\omega')$ we construct $\omega$ and $\omega'$ as follows.
  If the number of paths in $\In_L(\omega)$ and in $\In_L(\omega')$ are equal,
  \begin{itemize}
    \item we arbitrarily pair each path in $\In_L(\omega)$ with a path in $\In_L(\omega')$, and for each of these pairs, we construct thanks to Corollary \ref{cor:coincide_up_time_shift} the extensions of these paths so that they coincide eventually outside a (random) finite vertex set,
    \item we sample identically the trajectories of $\omega$ and $\omega'$ that do not touch $L$.
  \end{itemize}
 Otherwise, we construct $\Out_L(\omega)$ and $\Out_L(\omega')$ independently of each other.

 For this coupling, we have, for a large enough finite vertex set $M$, that $\omega$ and $\omega'$ coincide outside $M$ with probability not smaller than $1-\epsilon$. The proof is complete.
\end{proof}
  

\subsection{Proof in the case of tail-atomicity}\label{sec:proof_tail_atomicity}

We now prove Theorem \ref{th:tail_atomicity}. The proof relies on a decomposition of the interlacement process $\omega$ as a sum of (independent) Poisson point processes similar to the interlacement process under the assumption of tail-triviality. 

\begin{proof}[Proof of Theorem \ref{th:tail_atomicity}]
We assume that $\mathcal{T}_{MC}$ is purely atomic, so we consider a collection $\mathcal{C}$ of events in $\mathcal{T}_{MC}$ that, for arbitrary $x$ in $\mathbb{V}$, are atoms of $(W^+, \mathcal{T}_{MC}, P_x)$ and form a partition of $W^+$. 

For every $A, B \in \mathcal{C}$, let us denote by $\omega_{|A\to B}$ the restriction of $\omega$ to trajectories belonging to the event $``A\to B''$ (defined in the end of Section \ref{sec:intro_01laws}). 
Naturally, the random variable $\omega_{|A\to B}$ is a Poisson point process with intensity measure $\nu \ind_{A\to B}$. 
Let us now describe $\omega_{|A\to B}$ similarly as we did in Section \ref{sec:def_interl_process} for the complete interlacement process $\omega$. 
For every finite vertex set $K$, let
\begin{equation}
  e_{K, A\to B}(x):=
  \begin{cases}
    a_x P_x[A, \tau_K^+=\infty]  P_x[B] & \text{if } x\in  K,\\ 
    0 & \text{otherwise},
  \end{cases}
\end{equation}
\begin{equation}
  {\capa}_{A\to B}(K):=e_{K, A\to B}(\mathbb{V})= \sum_{x\in K} a_x P_x[A, \tau_K^+=\infty]  P_x[B],
\end{equation}
and
\begin{equation}
  \harm_{K, A\to B}(\cdot):= \frac{e_{K, A\to B}(\cdot)}{{\capa}_{A\to B}(K)}\, .
\end{equation}

To sample the trajectories of $\omega_{|A\to B}$ touching $K$:
\begin{itemize}
\item sample a number $\xi$ as a Poisson variable of parameter $\capa_{A\to B}(K)$,
\item independently for each $i=1,  \dots, \xi$, sample a point $x_i$ under $\harm_{K, A\to B}$,
\item independently for each $i=1, \dots, \xi$, attach to $x_i$ a trajectory that enters in $K$ through $x$, at time 0:
\begin{itemize}
  \item for the negative indices, sample a trajectory under $P_x[ \cdot |A, \tau_K^+=\infty]$ and apply the function $n\mapsto -n$ to the indices,
  \item for the positive indices, sample a trajectory under $P_x[ \cdot |B ]$.
\end{itemize}
\end{itemize}

We can also formulate this in terms of hinge process. Define 
\begin{equation}
h_{K, A\to B}(x, y)= e_{K, A\to B}(x) \times P_x[X_{\lambda_K}= y | B ] = a_x P_x[A, \tau_K^+=\infty]  P_x[B, X_{\lambda_K}= y] .
\end{equation}
To sample the trajectories of $\omega_{|A\to B}$ touching $K$,
\begin{itemize}
\item sample first the hinge process, as a Poisson point process with intensity $h_{K, A\to B}$.
\item to each hinge couple $(x, y)$, attach a trajectory $w^*$ which decomposes as $In_K(w)$ and $Out_K(w)=(Out_{K\rightarrow} (w), Out_{\rightarrow K}(w))$:
\begin{itemize}
\item $In_K(w)$ is sampled according to $P_x[ (X_0, X_1, \dots, X_{\lambda_K})\in \cdot | B, X_{\lambda_K}=y]$ 
\item $Out_{K\rightarrow} (w)$ is obtained by applying a time reversal to a trajectory sampled according to $P_x[ \cdot | A, \tau^+_K=\infty]$,
\item $Out_{\rightarrow K}(w)$ is sampled according to $P_y[ \cdot | B, \tau^+_K=\infty]$.
\end{itemize} 
\end{itemize}
For $K\subset L$, sampling $\In_L(\omega_{A\to B})$ given $\In_K(\omega_{A\to B})$ is also very similar to what we presented in Section \ref{sec:hinge_decomp}. 
The crucial observation is that the past pieces of trajectory are sampled under the Markov chain $X$ conditioned on $A$ while the future pieces are sampled under the Markov chain $X$ conditioned on $B$. Since $A$ and $B$ are atoms of $\mathcal{T}_{MC}$, we note that the two Markov chains obtained after these conditionings are tail-trivial, so that Corollary \ref{cor:coincide_up_time_shift} applies.

\smallskip

We now conclude the proof of Theorem \ref{th:tail_atomicity}. 
Consider $A, B \in \mathcal{C}$. Assume that $\nu(A\to B)$ is infinite. Then we reproduce the proof of Theorem \ref{th:tail_triviality} in Section \ref{sec:proof_tail_triviality}, to show that $\omega_{|A\to B}$ satisfies the 0-1 law under $\mathbb{P}$: for a given vertex set $K$ and a given admissible finite path inside $K$, we want to construct a coupling of two copies of $\omega_{|A\to B}$, one conditioned on having exactly this path inside $K$ and one on having no trajectories touching $K$, such that they coincide outside a large vertex set with large probability. We proceed as follows:
\begin{itemize}
\item since $\capa_{A\to B}(K)\underset{K \to \mathbb{V} }\longrightarrow \nu(A\to B) = \infty$, we find a large enough finite vertex set $L$ and make the number of trajectories hitting $L$ in both copies of $\omega_{|A\to B}$ coincide with large probability,
\item we arbitrarily pair the trajectories touching $L$ between both copies of $\omega_{|A\to B}$ and use Corollary \ref{cor:coincide_up_time_shift} for the Markov chain $X$ conditioned respectively on $A$ and on $B$ to make them coincide outside a large vertex set $M$, with large probability,
\item we sample the trajectories that do not touch $L$ identically in both copies of $\omega_{|A\to B}$.
\end{itemize}
We deduce that the variable $\omega_{|A\to B} $ satisfies the 0-1 law under $\mathbb{P}$.

Now, if $\nu(A\to B)$ is infinite for every $A, B \in \mathcal{C}$, then all the variables $(\omega_{|A\to B})_{A,B  \in \mathcal{C} }$ satisfy the 0-1 law under $\mathbb{P}$, and so does their sum $\omega$.
This concludes the proof of Theorem \ref{th:tail_atomicity}.
\end{proof}

\begin{remark}
We stress that the variables $(\omega_{|A\to B})_{A,B  \in \mathcal{C} }$ form an independent family, although this does not play a role in the proof. 
\end{remark}

\section{Quantitative approach}\label{sec:quantitative}

\subsection{Preliminaries}

\subsubsection{Necessary and sufficient condition}\label{sec:necessary_and_sufficient_for_quantitative}

For any finite vertex set $K$, 
\begin{itemize}
\item we call admissible hinge couple on $K$ any couple $(x,y)\in \partial_X K \times \partial_X K $;
\item we call admissible hinge configuration on $K$ any configuration $\rho$ such that $\p(\Hin_K=\rho)>0$, that is, any configuration of the form $\sum_{i=1}^k \delta_{(x_i,y_i)}$, where $k$ is a finite integer (possibly zero), and for each $i=1, \dots, k$, $(x_i,y_i)$ is an admissible hinge couple on $K$. 
\end{itemize}

\begin{lemma}\label{lem:criterion_dTV_to_0_law}
  The 0-1 law for $\mathcal{T}_{RI}$ holds if and only if for every finite vertex set $K$, for every admissible hinge configuration $\rho$ on $K$, as $L\to \mathbb{V}$,
  \begin{equation}\label{eq:criterion_d_TV_cond}
    d_{TV}\left(\mathbb{P}(\Hin_L \in \cdot | \Hin_K=\rho),
    \mathbb{P}(\Hin_L \in \cdot | \Hin_K=0) \right) \to 0.
  \end{equation}
  Furthermore, the condition can be restricted to $\rho$'s that consist of a single hinge couple.
\end{lemma}

\begin{proof}
  We have already shown that \eqref{eq:criterion_d_TV_cond} implies the 0-1 law for $\mathcal{T}_{RI}$: this is Lemma \ref{lem:suf_cond_0_1_law_distance}, together with Remark \ref{rem:rewrite_d_TV_cond}. 
  
  To show the converse implication, we use Lemma \ref{lem:general_equiv_01law}. Assume that the 0-1 law for $\mathcal{T}_{RI}$ holds and consider a finite vertex set $K$, an admissible hinge configuration $\rho$ on $K$, and an exhaustion $(L_n)_{n\ge 0}$ of $\mathbb{V}$. Assume without loss of generality that $L_0=\emptyset$, $L_1=K$. Apply Lemma \ref{lem:general_equiv_01law} to the time-inhomogeneous Markov chain $(\Hin_{K_n})_{n\ge 0}$ to obtain:
  \begin{equation}
    d_{TV}\left(\mathbb{P}(\Hin_{L_n} \in \cdot | \Hin_K=\rho),
    \mathbb{P}(\Hin_{L_n} \in \cdot | \Hin_K=0) \right) 
    \underset{n\to\infty}\longrightarrow 0.
  \end{equation}
  This concludes the proof of the lemma.
\end{proof}


\subsubsection{Distance between Poisson point processes}

If $\nu$ is a finite measure on a finite set $\mathcal{X}$, let us denote by $PPP(\nu)$ the law of a Poisson point process on $\mathcal{X}$ with intensity measure $\nu$. If furthermore, $\pi$ is a probability measure on $\mathcal{X}$, let us denote by $PPP(\nu) \oplus \pi$ the law of the superposition of this same Poisson point process and a random point in $\mathcal{X}$ sampled independently under $\pi$. 

\begin{lemma}\label{lem:general_cond_dTVto0_for_PPP}
For every $n\ge 1$ we consider a finite set $\mathcal{X}_n$, and two finite measures $\nu_n$ and $\pi_n$ on $\mathcal{X}_n$, $\pi_n$ being a probability measure. Then, 
\begin{equation}
  d_{TV}(PPP(\nu_n) \oplus \pi_n  , PPP(\nu_n) ) \underset{n\to\infty}{\longrightarrow} 0
\end{equation}
if and only if, for every $\epsilon>0$, 
\begin{equation}
  \sum_{x\in \mathcal{X}_n} 
  \left(\pi_n(x)-\epsilon \nu_n(x)\right)_+  
  \underset{n\to\infty}{\longrightarrow} 0\, .
\end{equation}
\end{lemma}


\begin{proof}
  We consider a finite state space $\mathcal{X}$, and two measures $\nu$ and $\pi$ on $\mathcal{X}$, $\pi$ being a probability measure. 

  Let us first show, for every  $\epsilon>0$ and $a>0$, that if $\sum_{x\in \mathcal{X}} 
  \left(\pi(x)-\epsilon \nu(x)\right)_+  \le a$, then  
\begin{equation}\label{eq:dTV_PPP_upp}
  d_{TV}( PPP(\nu) \oplus \pi, PPP(\nu)) \le \frac{1}{2}\sqrt{\frac{\epsilon}{1-a}} + a ,
\end{equation}
To prove this, we construct $PPP(\nu)$ using an auxiliairy Poisson point process. We consider $\sum_{i\ge 1} \delta{(x_i, u_i)}$ a Poisson point process on $\mathcal{X} \times [0, \infty)$ with intensity measure $\rho=\mathbf{c} \otimes du$, where $\mathbf{c}:=\sum_{x\in \mathcal{X}} \delta_x$ is the counting measure on $\mathcal{X}$ and $du$ is the Lebesgue measure on $(0, \infty)$. 
Then 
$\omega_\nu:=\sum_{i\ge 1: u_i\le \nu(x_i)} \delta_{x_i}$ 
realises $PPP(\nu)$. 
Furthermore, let us consider the domain
$$D=\left\{(x,u)\in \mathcal{X}\times (0, \infty): u\le \frac{\pi(x)}{\epsilon} \wedge \nu(x)\right\}.$$
Let us tilt $\omega_\nu$ only by changing the law of the number of points falling in $D$: let us make it follow the law of $Poi(\rho(D))+1$ instead of the law $Poi(\rho(D))$. 
Then, we obtain a point process which is the superposition of a $PPP(\nu)$ and a point sampled under the renormalisation of the measure $\frac{\pi(\cdot)}{\epsilon} \wedge \nu(\cdot)$. This last distribution is at total-variation distance not larger than $a$ from $\pi$.
Thus:
\begin{equation}
  d_{TV}( PPP(\nu) \oplus \pi, PPP(\nu)) \le d_{TV}(Poi(\rho(D)), Poi(\rho(D))+1)+ a ,
\end{equation}
We conclude the proof of \eqref{eq:dTV_PPP_upp} by using Lemma \ref{lem:dTV_Poisson} and observing that $\rho(D) \ge \frac{1-a}{\epsilon}$.

\medskip

Let us furthermore show for every $\epsilon >0$ and $a>0$ the existence of a positive real $D_{\epsilon, a}$ such that if 
$\sum_{x\in \mathcal{X}} \left(\pi(x)-\epsilon \nu(x)\right)_+  \ge a$, then 
  \begin{equation}\label{eq:dTV_PPP_low}
  d_{TV}( PPP(\nu) \oplus \pi, PPP(\nu)) \ge D_{\epsilon, a}.
\end{equation}
Indeed, let us consider $A:=\{ x \in \mathcal{X}: \pi(x)  > \epsilon \nu(x)\}$. We have by assumption $\pi(A) \ge \epsilon \nu(A) +a$. Considering the number of points falling in $A$, we get the lower bound:
$$d_{TV}( PPP(\nu) \oplus \pi, PPP(\nu))
\ge d_{TV}( Poi(\nu(A))+Ber(\pi(A)), Poi(\nu(A))) 
\ge D_{\epsilon, a} \, , $$
where we denote by $Poi(\lambda)+Ber(p)$ the law of the independent sum of a Poisson variable of parameter $\lambda$ and a Bernoulli variable of parameter $p$ and we set 
$$D_{\gep, a}  := 
\inf_{\lambda \ge 0, \,  p\in [0, 1]:\, \epsilon \lambda +a\le p} 
d_{TV}(Poi(\lambda)+Ber(p), Poi(\lambda)).$$
We observe that the function 
$$ (\lambda,p) \mapsto  d_{TV}(Poi(\lambda)+Ber(p), Poi(\lambda))$$
defined on the compact   
$\{(\lambda, p): \lambda \ge 0, p\in [0, 1], \epsilon \lambda+a\le p \}$ 
is continuous and positive, hence $D_{\gep, a}>0$. 
Thus, \eqref{eq:dTV_PPP_low} is established and the lemma follows from \eqref{eq:dTV_PPP_upp} and \eqref{eq:dTV_PPP_low}.
\end{proof}

\subsection{Proof of the quantitative criterion}\label{sec:proof_quantitative}

We now prove Theorem \ref{th:quantitative_criterion}.

\begin{proof}[Proof of Theorem \ref{th:quantitative_criterion}]
Assume that the 0-1 law for $\mathcal{T}_{RI}$ holds under $\mathbb{P}$. Consider a vertex $x$ in $\mathbb{V}$, set $K=\{x\}$ and consider the hinge couple $(x,x)$ on $K$. Applying Lemma \ref{lem:criterion_dTV_to_0_law}, we get the convergence towards zero, as $L\to\mathbb{V}$, of the total variation distance between 
\begin{itemize}
  \item the law 
  $\mathbb{P}(\Hin_L \in \cdot | \Hin_K=0)$, that is, a Poisson point process with intensity measure 
  $h_{L}^{x}(x',y') := h_L(x',y') P_{x'}(\tau_x=\infty|X_{\lambda_L}=y')$,
  \item and the law 
  $\mathbb{P}(\Hin_L \in \cdot | \Hin_K=\delta_{(x,x)})$, that is, the law of the superposition of the above Poisson point process and an independent point sampled under the probability 
  $p_{L}^{x}(x', y'):= P_x(X_{\lambda_L}= x'| \tau_x^+=\infty) P_x(X_{\lambda_L}= y'| \tau_x^+=\infty)$.
\end{itemize}
Applying Lemma \ref{lem:general_cond_dTVto0_for_PPP} for any exhaustion $(L_n)_{n\ge 0}$ of $\mathbb{V}$, this implies that for every $\epsilon>0$, 
\begin{equation}
  \sum_{x',y' \in \partial_X L} 
  \left(p_{L}^{ x}(x', y') 
  - \epsilon h_{L}^{x}(x',y') \right)_+  \underset{L \to \mathbb{V}}\longrightarrow 0.
\end{equation}
Bounding $h_{L}^{x}$ by $h_L$, this implies that  
\begin{equation}
  \sum_{x',y'\in  \partial_X L} 
  \left(p_{L}^{x}(x', y')  - \epsilon h_L(x', y')   \right)_+  
  \underset{L \to \mathbb{V}}{\longrightarrow} 0.
\end{equation}
Now, by reversibility:
\begin{equation}
  a_x  P_x[X_{\lambda_L}=x', \tau_x^+=\infty]    P_x[X_{\lambda_L}=y']  = a_{x'} P_{x'}[\tau_L=\infty]  P_{x'}[\tau_x < \infty, X_{  \lambda_L}=y'].
\end{equation}
We also note that $P_x[X_{\lambda_L} = z,\tau_x^+=\infty]= P_x[\tau_x^+=\infty] P_x[X_{\lambda_L} = z]$, for every $z\in \mathbb{V}$.
It follows that:
\begin{equation}\label{eq:rewrite_proba_induced_hinges_singleton}
 p_{L}^x(x', y') =  \frac{h_L(x', y')}{a_x P_x(\tau_x^+=\infty) } 
  P_{x'}[\tau_x < \infty | X_{\lambda_L}=y'].
\end{equation}
Recalling that $x$ is fixed, and playing on $\epsilon$, we deduce that for every $\epsilon>0$
\begin{equation}
  \sum_{x',y' \in \partial_X L} 
  h_L(x', y')
  \left(P_{x'}[\tau_x < \infty | X_{\lambda_L}=y']
  - \epsilon  \right)_+  \underset{L \to \mathbb{V}}{\longrightarrow} 0.
\end{equation}
This is the condition of Theorem \ref{th:quantitative_criterion}.

\medskip 

Conversely, let us assume that the condition of Theorem \ref{th:quantitative_criterion} holds, and let us pick a finite vertex set $K$, and $(x,y)$ an admissible hinge couple on $K$. Following Lemma \ref{lem:general_cond_dTVto0_for_PPP}, our aim is to show the convergence towards zero, as $L\to\mathbb{V}$, of the total variation distance between
\begin{itemize}
  \item the law 
  $\mathbb{P}(\Hin_L \in \cdot | \Hin_K=0)$, that is, a Poisson point process with intensity measure 
  $h_{L}^{K}(x',y') := h_L(x',y') P_{x'}(\tau_K=\infty|X_{\lambda_L}=y')$
  \item and the law 
  $\mathbb{P}(\Hin_L \in \cdot | \Hin_K=\delta_{(x,y)})$, that is, the law of the superposition of the above Poisson point process and an independent point sampled under the probability 
  $p_{L}^{K, x,y} (x', y'):= P_x(X_{\lambda_L}= x'| \tau_K^+=\infty) P_y(X_{\lambda_L}= y'| \tau_K^+=\infty)$.
\end{itemize}
By Lemma \ref{lem:general_cond_dTVto0_for_PPP}, it is enough to establish that for every $\epsilon>0$, 
\begin{equation}\label{eq:sum_pLKxy-epshLK_to0}
  \sum_{x',y' \in \partial_X L} 
  \left(p_L^{K, x, y}(x',y') 
  - \epsilon h_L^{K}(x',y') \right)_+  \underset{L \to \mathbb{V}}{\longrightarrow} 0.
\end{equation}
By reversibility:
\begin{equation}
  \begin{aligned}
  a_x &  \p_x[X_{\lambda_L}=x', \tau_K^+=\infty]   \p_x[X_{\lambda_K}=y] \p_y[X_{\lambda_L}=y' | \tau_K^+=\infty] \\
   =  & a_{x'} \p_{x'}[\tau_L^+=\infty]
   \p_{x'}[\tau_K < \infty, X_{\tau_K} = x, X_{\lambda_K}=y, X_{\lambda_L}=y']
  \end{aligned}
\end{equation}
hence
\begin{equation}
  p_L^{K, x, y}(x',y')= \frac {h_L(x', y')}{ h_K(x, y)} 
  \p_{x'}[\tau_K < \infty, X_{\tau_K} = x, X_{\lambda_K}=y| X_{\lambda_L}=y'],
\end{equation}
and therefore, playing on $\epsilon$, \eqref{eq:sum_pLKxy-epshLK_to0} is equivalent to the fact that for every $\epsilon>0$, 
\begin{equation}
  \begin{aligned}
    \sum_{x', y'\in \partial_X L} &  h_L(x', y') \\
    &\times \big(  P_{x'}[\tau_K<\infty, X_{\tau_K}=x, X_{\lambda_K}=y |X_{\lambda_L}=y']  -\epsilon P_{x'}(\tau_K=\infty|X_{\lambda_L}=y') \big)_+ 
  \end{aligned}
  \end{equation}
vanishes as $L \to \mathbb{V}$. Bounding $P_{x'}[\tau_K<\infty, X_{\tau_K}=x, X_{\lambda_K}=y  |X_{\lambda_L}=y']$ by $P_{x'}[\tau_K<\infty|X_{\lambda_L}=y']$, the latter follows from the fact that for every $\epsilon>0$,
\begin{equation}
    \sum_{x',y' \in \partial_X L} h_L(x', y') \left((1+\epsilon) P_{x'}[\tau_K<\infty | X_{\lambda_L}=y']-\epsilon \right)_+ \underset{L \to \mathbb{V}}{\longrightarrow} 0,
  \end{equation}
which, finally, can be derived elementary from the condition of Theorem \ref{th:quantitative_criterion} by using the bound:
\begin{equation}
  P_{x'}[\tau_K<\infty|X_{\lambda_L}=y'] \le \sum_{z\in K}  P_{x'}[\tau_z<\infty|X_{\lambda_L}=y']
\end{equation}
and playing on $\epsilon$.
\end{proof}

\subsection{Proof of the weak 0-1 law}

This section is devoted to showing Theorem \ref{th:weak_01law}, that is, the 0-1 law for events that depend, for every finite vertex set $K$, only on the future pieces of trajectory outside $K$. 

Let us observe that in the case of tail-atomicity, one could reproduce our proof (in Section \ref{sec:proof_tail_atomicity}) to show that this 0-1 law holds as soon as $\nu(W^+ \to A)\in \{0, \infty\}$ for every $A\in \mathcal{T}_{MC}$. 
And in fact, for $A$ an atom of $\mathcal{T}_{MC}$ under $P_x$ (for arbitrary $x$), applying Lemma \ref{lem:cap_to_infty} to the Markov chain $X$ conditioned on $A$ yields that $\nu (A \to A)=\lim_{K\to\mathbb{V}} \capa_{A\to A}(K)= \infty$ and a fortiori $\nu(W^+ \to A)=\infty$. 

\smallskip

Since we aim to prove Theorem \ref{th:weak_01law} in full generality, we will rather rely on the quantitative approach of the present Section \ref{sec:quantitative}. 
We will proceed in two steps. 
We will first prove that a condition similar to that of Theorem \ref{th:quantitative_criterion} holds in full generality. Then, we will prove Theorem \ref{th:weak_01law}, using the same scheme of proof as in Section \ref{sec:proof_quantitative}.

\begin{lemma}\label{lem:quantitative_weak}
  For every site $x\in \mathbb{V}$, for every $\epsilon>0$, as $L \to \mathbb{V}$,
  \begin{equation}
    \sum_{x'\in L} e_L(x') \left(P_{x'}[\tau_x < \infty]-\epsilon \right)_+ \to 0.
  \end{equation}
\end{lemma}

\begin{remark} \label{rem:rewrite_quantitative_weak}
  We stress that, due to \eqref{eq:consistency_eq_meas},
\begin{equation}
    \sum_{x'\in L} e_L(x') P_{x'}[\tau_x < \infty]= Cap(\{x\}).
  \end{equation}
The content of the lemma can therefore be interpreted as the fact that this sum is ``diluted'' enough, that there is not too much mass supported on some vertices $x$.  

Using the bounds $b \ind_{a > 2b} \le (a-b)_+ \le \ind_{a > b}$ for $a \in [0, 1], b \ge 0$, the condition in the theorem can be equivalently formulated as follows: for every site $x\in \mathbb{V}$, for every $\epsilon>0$, as $L\to \mathbb{V}$,
  \begin{equation}
    \sum_{x'\in  L} e_L(x') \ind_{P_{x'}[\tau_x < \infty ] >\epsilon} = e_L(\{x'\in L : P_{x'}[\tau_x < \infty ] >\epsilon \} )\to 0.
  \end{equation}
\end{remark}

\begin{proof}[Proof of Lemma \ref{lem:quantitative_weak}] 
  Let us fix $x\in \mathbb{V}$ and $\epsilon >0$. We consider the set 
  \[D:=\{y \in \mathbb{V} : P_y[\tau_x<\infty]> \epsilon\}.\]
  As observed in Remark \ref{rem:rewrite_quantitative_weak}, our aim is to show that $e_L(D)$ vanishes as $L\to\mathbb{V}$.
  
  Of course, if $D$ is finite this is immediate: as soon as $L$ contains $\overline{D}: =\{x\in \mathbb{V}:\exists y \in D, x\sim y\}$, we have $\partial_X L \cap D =\emptyset$ and thus $e_L(D)=0$. But $D$ can be infinite. Think for example of a simple random walk on $\mathbb{Z}$ with a positive drift and $x=0$: in this example, $P_y[\tau_x<\infty]=1$ for every $y\le 0$.

  Let us show that $D$ is however necessarily transient, i.e., that for every $z\in \mathbb{V}$, 
  $P_z[X_n \text{ visits } D \text{ infinitely many times}]=0$. 
  Our proof relies on the fact that $X$ visits $x$ almost surely finitely many times, due to transience.

  For any $y$ in $D$, we pick an integer $t_y$ such that $P_y[\tau_x <t_y]\ge \frac{\epsilon}{2}$. Then we decompose the trajectory of the Markov chain $X$ starting from $z$ in disjoint excursions:
  \begin{itemize}
  \item starting from $z$, we wait for the first visit to $D$, say at time $N_1$ and at site $Y_1$; the first excursion is the piece of trajectory from time $N_1$ to time $N_1+t_{Y_1}-1$;
  \item we reiterate: for $n\ge 1$, after the $n$-th excursion, we wait for the first visit to $D$, say at time $N_{n+1}$ and site $Y_{n+1}$, then we define the $(n+1)$-th excursion to be the piece of trajectory from time $N_{n+1}$ to time $N_{n+1}+t_{Y_{n+1}}-1$. 
  \end{itemize}
  There may be only finitely many excursions, meaning that at some point the trajectory stops visiting $D$; this is actually what we want to show happens almost surely.
  We observe that every excursion has a probability at least $\frac{\epsilon}{2}$ of visiting $x$. Thus, the expected number of excursions is not larger than $\frac{2}{\epsilon}$ times the expected number of visits to $x$, which is finite by transience of $X$. We deduce that the number of excursions is almost surely finite, hence $D$ is transient.

  To conclude, we observe, using the definition of $D$ and reversibility that
  \begin{align*}
  e_L(D) & =    \sum_{y\in D} a_y P_y[\tau_L^+=\infty]\\
    & \le \frac{1}{\epsilon} \sum_{y\in D} a_y P_y[\tau_L^+=\infty]  P_y[\tau_x<\infty] \\
    & = \frac{1}{\epsilon} \sum_{y\in D} a_x P_x[X_{\lambda_L}=y]
     = \frac{a_x}{\epsilon} P_x[X_{\lambda_L}\in D].
  \end{align*}
  The transience of $D$ implies that $P_x[X_{\lambda_L}\in D]$ vanishes as $L\to\mathbb{V}$ and the proof is complete.
\end{proof}

\begin{proof}[Proof of Theorem \ref{th:weak_01law}]
  For convenience, we prefer to show the 0-1 law for the $\sigma$-algebra of non-local backward events, mentioned in Remark \ref{rem:backward_weak} and defined by focusing on the past pieces of trajectory $Out_{\rightarrow K}(w)$ instead of the future pieces $Out_{K\rightarrow}(w)$. As justified in that remark, these two ``weak'' 0-1 laws are equivalent.

  Analogously to Lemma \ref{lem:criterion_dTV_to_0_law} (or Lemma \ref{lem:suf_cond_0_1_law_distance}), it is enough to show that for every finite vertex set $K$, for every $x\in \partial_X K$, as $L\to\mathbb{V}$,
  \begin{equation}\label{eq:criterion_d_TV_cond_bis}
    d_{TV}\left(\mathbb{P}(\Hin_{\to L} \in \cdot | \Hin_{\to K}=\delta_x),
    \mathbb{P}(\Hin_{\to L} \in \cdot | \Hin_{\to K}=0) \right) \to 0.
  \end{equation}
  where the random variable $\Hin_{\to K}$ is defined as
  \begin{equation}
    \Hin_{\to K}(\omega)= \sum_{i\ge 1} \delta_{X_{\tau_K(w_i)}}
  \end{equation}
  if $\omega = \sum_{i\ge 1} \delta_{w_i^*}$ and  $w_i$ is a representative of $w_i^*$ for each $i\ge 1$. 
  Now,
  \begin{itemize} 
    \item  $\mathbb{P}(\Hin_{\to L} \in \cdot | \Hin_{\to K}=0)$ is the law of Poisson point process on $\mathbb{V}$ with intensity measure $e_L^K (x') := e_L(x')P_{x'}[\tau_K=\infty]$,
    \item while $\mathbb{P}(\Hin_{\to L} \in \cdot | \Hin_{\to K}=\delta_x)$ is the law of the superposition of this Poisson point process and an independent point in $\mathbb{V}$ sampled under $p_L^{K, x} (x') := P_x[X_{\lambda_L}=x'|\tau_K^+=\infty]$.
  \end{itemize}
  By Lemma \ref{lem:general_cond_dTVto0_for_PPP}, it is enough to establish that for every $\epsilon>0$, 
\begin{equation}\label{eq:sum_pLKxy-epshLK_to0_weak}
  \sum_{x'  \in  L} 
  \left(p_L^{K, x}(x') 
  - \epsilon \, e_L^{K}(x') \right)_+ \underset{L \to \mathbb{V}}{\longrightarrow}.
\end{equation}
By reversibility:
\begin{equation}
  a_x \p_x[X_{\lambda_L}=x', \tau_K^+=\infty]  = a_{x'} \p_{x'}[\tau_L^+=\infty] 
   \p_{x'}[\tau_K < \infty, X_{\tau_K} = x]
\end{equation}
hence
\begin{equation}
  p_L^{K, x}(x')= \frac {e_L(x')}{ e_K(x)} 
  \p_{x'}[\tau_K < \infty, X_{\tau_K} = x],
\end{equation}
and therefore, playing on $\epsilon$, \eqref{eq:sum_pLKxy-epshLK_to0_weak} is equivalent to the fact that for every $\epsilon>0$, 
\begin{equation}
    \sum_{x'\in   L}  e_L(x')  \big(  P_{x'}[\tau_K<\infty, X_{\tau_K}=x]  -\epsilon P_{x'}(\tau_K=\infty) \big)_+  \underset{L \to \mathbb{V}}{\longrightarrow}.
\end{equation}
Bounding $P_{x'}[\tau_K<\infty, X_{\tau_K}=x]$ by $P_{x'}[\tau_K<\infty]$, the latter follows from the fact that for every $\epsilon>0$,
\begin{equation}
    \sum_{x'\in  L} e_L(x') \left((1+\epsilon) P_{x'}[\tau_K<\infty]-\epsilon \right)_+  \underset{L \to V}\longrightarrow 0,
  \end{equation}
which, finally, can be derived elementary from Lemma \ref{lem:quantitative_weak} by using the bound:
\begin{equation}
  P_{x'}[\tau_K<\infty] \le \sum_{z\in K}  P_{x'}[\tau_z<\infty]
\end{equation}
and playing on $\epsilon$.
\end{proof}

\subsection{Proof of the 0-1 law for increasing events}

In this section, we prove Theorem \ref{th:increasing_01law}. We begin with a lemma.

\begin{lemma}\label{lem:for_01_increasing}
  For every finite vertex set $K$, for every admissibe path $\gamma$ inside $K$, for every increasing event $E$ belonging to $\widetilde{\mathcal{T}}_{RI}$, we have:
  \begin{equation}
    \p(E | \In_K=\delta_\gamma) \le \p(E| \In_K=0).
  \end{equation}
\end{lemma}

\begin{proof}[Proof of Lemma \ref{lem:for_01_increasing}]
We consider a finite vertex set $K$, and an admissibe path $\gamma$ inside $K$, say with extremities $(x,y)\in \partial_X K\times \partial_X K$. Let furthermore $\epsilon >0$. We are going to construct a vertex set $L$ (containing $K$) and, under a probability denoted $\mathbf{P}$, a coupling $(\omega , \omega')$ of $\p( \cdot | \In_K=\delta_\gamma) \le \p( \cdot | \In_K=0)$ together with random trajectories $w^*, w_1^*, w_2^*\in W^*$ such that, with probability not smaller than $1-\epsilon$: 
\begin{itemize}
  \item $w^*$ belongs to $\omega$ and is precisely the trajectory of $\omega$ that visits $K$, $w_1^*$ and $w_2^*$ belong to $\omega'$;
  \item $\omega- \delta_{w^*} = \omega'-\delta_{w_1^*}-\delta_{w_2^*}$;
  \item $Out_{\to L}(w)=Out_{\to L}(w_1) $, i.e., $w^*$ and $w_1^*$ coincide up to their entry in $L$, in particular their entry points in $L$ coincide;
  \item $Out_{L \to}(w)=Out_{L \to}(w_2) $, i.e., $w^*$ and $w_2^*$ coincide from their last visit to $L$, in particular their last visit points in $L$ coincide;
\end{itemize}
where $w, w_1$ and $w_2$ are representatives of $w^*$, $w_1^*$ and $w_2^*$ respectively. 

\smallskip

To construct the coupling, we use twice the result stated in \eqref{eq:criterion_d_TV_cond_bis} and established during the proof of Theorem \ref{th:weak_01law}, which we formulate in this way: for a large enough $L$, there is a way, given a Poisson point process $\omega_{entry}$ under $ \p( \Hin_{\to L} \in \cdot | \In_K=0)$ to pick a point $X$ in $\omega_{entry}$ such that the couple $(\omega_{entry}-\delta_X, X)$ is at total variation distance at most $\epsilon/2$ from the product law $ \p( \Hin_{\to L} \in \cdot | \In_K=0) \otimes p_L^{K,x}$, and similarly with $x$ replaced by $y$.

So, consider the Poisson point process $\omega'$ distributed under $\p( \cdot | \In_K=0)$. Among the entry points to $L$ of the trajectories in $\omega'$, we pick a point $X$ such as prescribed above, and denote the corresponding trajectory $w_1^*$. Among the exit points of $\omega' - \delta_{w_1^*}$, also pick a point $Y$ such as described above and denote the corresponding trajectory $w_2^*$. Furthermore, sample (under the law of the MC $X$) a trajectory inside $L$ that links $X$ and $Y$ and leaves trace $\eta$ on $K$. Set $w^*$ to be the trajectory that coincides with $w_1^*$ upp to its entry in $L$, is given inside $L$ by the previously sampled trajectory, and coincides with $w_2^*$ after its last visit to $L$. Finally set $\omega$ to be $\omega'-\delta_{w_1^*} - \delta_{w_2^*} + \delta_{w^*}$. 




\medskip

Let us finally derive the lemma. For every realisation of the coupling, we pick (arbitrarily) a trajectory $w_{add} \in W_L$ such that $Out_{\to L}(w_{add})=Out_{\to L}(w_2) $ and $Out_{L \to}(w_{add})=Out_{L \to}(w_1) $ and we set $w_{add}^* =\pi(w_{add})$. 
We observe that, when the coupling is successful, we have $\widetilde{\Out}_L(\omega +\delta_{w_{add}^*})= \widetilde{\Out}_L(\omega')$. It follows that for every increasing event $E$ belonging to $\widetilde{\mathcal{T}}_{RI}$,
\begin{equation}
  \begin{aligned}
    \p(E | \In_K=\delta_\gamma) 
    & = \mathbf{P}(\omega  \in E)\\
    & \le \mathbf{P}(\omega +\delta_{w_{add}^*} \in E)\\
    & \le \mathbf{P}(\omega' \in E) + \epsilon\\
    & = \p(E| \In_K=0) +\epsilon.
  \end{aligned}
\end{equation}
Since $\epsilon$ is arbitrary, the lemma follows.
\end{proof}

\begin{proof}[Proof of Theorem \ref{th:increasing_01law}]
  From Lemma \ref{lem:for_01_increasing}, we derive that for every finite vertex set $K$, for every admissible configuration $\eta$ inside $K$, for every admissibe path $\gamma$ inside $K$, for every increasing event $E$ belonging to $\widetilde{\mathcal{T}}_{RI}$, we have:
  \begin{equation}\label{eq:extend_for_01_increasing}
    \p(E | \In_K=\eta + \delta_\gamma) \le \p(E| \In_K=\eta).
  \end{equation}
  This relies on the observation that if $E$ is increasing and belongs to  $\widetilde{\mathcal{T}}_{RI}$, then for every (finite) configuration $\sum_{i=1}^k \delta_{w_i^*}$, with $k\in \N$, and $w_i^*\in W^*$, the event
  \begin{equation}
    \left\{ \omega \in \Omega : \omega + \sum_{i=1}^k \delta_{w_i^*} \in E \right\}
  \end{equation} 
  is also increasing and also belongs to $\widetilde{\mathcal{T}}_{RI}$. Just apply this fact to the configuration arising from $\eta$ to derive \eqref{eq:extend_for_01_increasing} from Lemma \ref{lem:for_01_increasing}.

  By induction, it follows from \eqref{eq:extend_for_01_increasing} that for every finite vertex set $K$, for every admissible configuration $\eta$ inside $K$, for every increasing event $E$ belonging to $\widetilde{\mathcal{T}}_{RI}$, we have:
  \begin{equation}\label{eq:extend_encore_for_01_increasing}
    \p(E | \In_K=\eta ) \le \p(E| \In_K= 0 ).
  \end{equation}
  We finally observe that the converse inequality holds, since $\p( \cdot | \In_K=\eta ) $ dominates $\p(\cdot | \In_K= 0 )$. So we have the equality in \eqref{eq:extend_encore_for_01_increasing} and we deduce that any increasing event $E$ belonging to $\widetilde{\mathcal{T}}_{RI}$ is independent of $\In_K$ for every finite vertex set $K$, and therefore is independent of $\mathcal{A}$, thus also of itself, so that $\p(E)\in \{0, 1\}$. 
\end{proof}

\bibliographystyle{plainnat}
\bibliography{bibliography}

@article{Teixeira2009,
  title = {Interlacement percolation on transient weighted graphs},
  volume = {14},
  ISSN = {1083-6489},
  url = {http://dx.doi.org/10.1214/EJP.v14-670},
  DOI = {10.1214/ejp.v14-670},
  number = {none},
  journal = {Electronic Journal of Probability},
  publisher = {Institute of Mathematical Statistics},
  author = {Teixeira,  Augusto},
  year = {2009},
  month = jan 
}

@article{Sznitman2010,
  title = {Vacant set of random interlacements and percolation},
  volume = {171},
  ISSN = {0003-486X},
  url = {http://dx.doi.org/10.4007/annals.2010.171.2039},
  DOI = {10.4007/annals.2010.171.2039},
  number = {3},
  journal = {Annals of Mathematics},
  publisher = {Annals of Mathematics},
  author = {Sznitman,  Alain-Sol},
  year = {2010},
  month = apr,
  pages = {2039–2087}
}

@article{Griffeath1975,
  title = {A maximal coupling for {Markov} chains},
  volume = {31},
  ISSN = {1432-2064},
  url = {http://dx.doi.org/10.1007/BF00539434},
  DOI = {10.1007/bf00539434},
  number = {2},
  journal = {Zeitschrift f\"ur Wahrscheinlichkeitstheorie und Verwandte Gebiete},
  publisher = {Springer Science and Business Media LLC},
  author = {Griffeath,  David},
  year = {1975},
  pages = {95–106}
}

@article{Pitman1976,
  title = {On coupling of {Markov} chains},
  volume = {35},
  ISSN = {1432-2064},
  url = {http://dx.doi.org/10.1007/BF00532957},
  DOI = {10.1007/bf00532957},
  number = {4},
  journal = {Zeitschrift f\"ur Wahrscheinlichkeitstheorie und Verwandte Gebiete},
  publisher = {Springer Science and Business Media LLC},
  author = {Pitman,  J. W.},
  year = {1976},
  pages = {315–322}
}

@book{Lawler2010,
  title = {Random Walk: A Modern Introduction},
  ISBN = {9780511750854},
  url = {http://dx.doi.org/10.1017/CBO9780511750854},
  DOI = {10.1017/cbo9780511750854},
  publisher = {Cambridge University Press},
  author = {Lawler,  Gregory F. and Limic,  Vlada},
  year = {2010},
  month = jun 
}

@article{Collin2024,
  title = {Two-dimensional random interlacements: 0-1 law and the vacant set at criticality},
  volume = {169},
  ISSN = {0304-4149},
  url = {http://dx.doi.org/10.1016/j.spa.2023.104272},
  DOI = {10.1016/j.spa.2023.104272},
  journal = {Stochastic Processes and their Applications},
  publisher = {Elsevier BV},
  author = {Collin,  Orphée and Popov,  Serguei},
  year = {2024},
  month = mar,
  pages = {104272}
}

@book{Last2017,
  title = {Lectures on the Poisson Process},
  ISBN = {9781107458437},
  url = {http://dx.doi.org/10.1017/9781316104477},
  DOI = {10.1017/9781316104477},
  publisher = {Cambridge University Press},
  author = {Last,  G\"{u}nter and Penrose,  Mathew},
  year = {2017},
  month = oct 
}

\end{document}